\documentclass[titlepage,12pt]{article}
\usepackage{hyperref}
\usepackage{graphicx}
\usepackage[dvipsnames]{xcolor}
\usepackage{amssymb,amsthm,amsmath}
\usepackage[a4paper]{geometry}
\usepackage{datetime2}
\usepackage[utf8]{inputenc}
\usepackage[italian,english]{babel}

\date{}

\newtheorem{thm}{Theorem}[section]

\newtheorem{rmk}[thm]{Remark}
\newtheorem{prop}[thm]{Proposition}

\newtheorem{cor}[thm]{Corollary}
\newtheorem{ex}[thm]{Example}

\newcommand{\re}{\mathbb{R}}
\newcommand{\al}{a_{\lambda}}

\newcommand{\cl}{c_{\lambda}}
\newcommand{\ul}{u_{\lambda}}
\newcommand{\ssl}{s_{\lambda}}
\newcommand{\tl}{t_{\lambda}}

\newcommand{\Ek}{E_{\mathrm{kow}}}
\newcommand{\Eh}{E_{\mathrm{hyp}}}
\newcommand{\Et}{E_{\mathrm{Tar}}}
\newcommand{\G}{\mathcal{G}}
\newcommand{\Phih}{\Phi_{\mathrm{hyp}}}
\newcommand{\Phit}{\Phi_{\mathrm{Tar}}}
\newcommand{\ep}{\varepsilon}
\newcommand{\osc}{\operatorname{osc}}

\title{On the derivative loss for a class of abstract evolution equations with time dependent oscillating coefficients vanishing at initial time}

\author{Marina Ghisi\vspace{1ex}\\
{\normalsize Università degli Studi di Pisa} \\
{\normalsize Dipartimento di Matematica}\\
{\normalsize PISA (Italy)}\\
{\normalsize e-mail: \texttt{marina.ghisi@unipi.it}}
\and
Massimo Gobbino\vspace{1ex}\\
{\normalsize Università degli Studi di Pisa} \\
{\normalsize Dipartimento di Matematica}\\
{\normalsize PISA (Italy)}\\
{\normalsize e-mail: \texttt{massimo.gobbino@unipi.it}}
\and
Michael Reissig\vspace{1ex}\\
{\normalsize TU Bergakademie Freiberg}\\
{\normalsize Institute of Applied Analysis}\\
{\normalsize FREIBERG (Germany)}\\
{\normalsize e-mail: \texttt{reissig@math.tu-freiberg.de}}
}

\begin{document}

\maketitle

\begin{abstract}

We consider an abstract wave-type evolution equation with a propagation speed depending only on time. We assume that this propagation speed is the product of a monotone, but possibly degenerate, shape function and an oscillating factor bounded from above and from below by positive constants. We investigate how degeneracy and oscillations interact in determining the derivative loss of solutions.

After reducing the problem to a family of ordinary differential equations through spectral analysis, we derive frequency-dependent energy estimates. In the purely degenerate case, monotonicity of the shape function prevents excessive loss, and solutions lose at most one derivative. In the presence of oscillations, the resulting derivative loss is described by a minimization problem in which the degenerate and oscillatory effects genuinely interact.

The proof is based on a decomposition of the time interval into frequency-dependent regions, where different adapted energies are used.

\vspace{6ex}

\noindent{\bf Mathematics Subject Classification 2020 (MSC2020):}
Primary 35L80; Secondary 35L10, 35L90, 35B65, 35B40.


\vspace{6ex}


\end{abstract}


\setcounter{equation}{0}
\section{Introduction}

In this paper we consider the abstract evolution equation
\begin{equation}
    u''(t)+\gamma(t)^2Au(t)=0,
    \label{eqn:basic}
\end{equation}
where $T$ is a positive real number, $\gamma:(0,T)\to\re$ is a given function, and $A$ is a nonnegative self-adjoint operator on a Hilbert space $H$. This equation is an abstract model of the wave equation
\begin{equation}
    u_{tt}(t,x)-\gamma(t)^2\Delta u(t,x)=0,
    \label{eqn:wave}
\end{equation}
where $\Delta$ is the Laplacian in the whole space $\re^d$, or in an open set $\Omega\subseteq\re^d$ with suitable boundary conditions, for example of Dirichlet type. In this physical model the time-dependent coefficient $\gamma(t)$ represents a time-dependent propagation speed.

When $\gamma$ is a function of bounded variation, bounded from below by a positive constant, then \eqref{eqn:basic} and \eqref{eqn:wave} are well behaved, in the sense that solutions exist in the classical energy spaces. When one or both of these conditions fail, solutions may exhibit the so-called \emph{derivative loss}, meaning that the space where the solution lies for $t>0$ is weaker than the space where the initial data are chosen. This phenomenon was discovered in the seminal papers \cite{1970-Oleijnik,1974-IvrijPetkov,1979-DGCS,1983-CJS}, and then investigated deeply in the literature.

In this paper we consider propagation speeds of the form
\begin{equation}
    \gamma(t)=b(t)c(t),
    \label{eqn:gamma=b*c}
\end{equation}
where $b:[0,T]\to[0,+\infty)$ is nondecreasing, but possibly vanishes at $t=0$, or even in an interval of the form $[0,\tau]$, while $c:[0,T]\to\re$ is bounded from above and from below by positive constants. In other words, $b$ is responsible for the degenerate character of the equation, whereas $c$ is responsible for the oscillatory behavior of $\gamma$. Our aim is to quantify how these two effects interact, and how their interaction determines the derivative loss of solutions.

\paragraph{\textmd{\textit{Reduction to a family of ODEs}}}

From the technical point of view, Fourier or spectral analysis reduces the problem to the study of the family of ordinary differential equations
\begin{equation}
    \ul''(t)+\lambda^2\gamma(t)^2\ul(t)=0.
    \label{eqn:ODE}
\end{equation}

This is the scalar equation satisfied by each Fourier component, or more generally by each component corresponding to the spectral decomposition of the operator $A$.

For every $\lambda\geq 0$ and every solution $\ul$ to \eqref{eqn:ODE}, we consider the classical energy
\begin{equation}
    E_\lambda(t):=\ul'(t)^2+\lambda^2 \ul(t)^2,
    \nonumber
\end{equation}
where the dependence on $\ul$ is omitted from the notation, and we introduce the quantity
\begin{equation}
    \G(\lambda,T):=\log\left(\sup\left\{
    E_\lambda(t):\ \ul\text{ solves \eqref{eqn:ODE}},\ E_\lambda(0)\leq 1,\ t\in[0,T]
    \right\}\strut\right).
    \nonumber
\end{equation}

By definition, $\G(\lambda,T)$ is the best exponent $C$ for which all solutions to \eqref{eqn:ODE} satisfy an estimate of the form
\begin{equation}
    \ul'(t)^2+\lambda^2\ul(t)^2
    \leq\exp(C)\left[\ul'(0)^2+\lambda^2\ul(0)^2\right]
    \qquad
    \forall t\in[0,T].
    \nonumber
\end{equation}

Classically, the behavior of $\G(\lambda,T)$ as $\lambda\to +\infty$ determines the derivative loss of solutions to \eqref{eqn:basic} or \eqref{eqn:wave} according to the following scheme:
\begin{align*}
    \G(\lambda,T)\text{ bounded}
    &\quad\leadsto\quad
    \text{no derivative loss},
    \\[0.5ex]
    \G(\lambda,T)\to +\infty\quad\text{and}\quad\frac{\G(\lambda,T)}{\log\lambda}\to 0
    &\quad\leadsto\quad
    \text{arbitrarily small derivative loss},
    \\[0.5ex]
    \G(\lambda,T)\simeq\log\lambda
    &\quad\leadsto\quad
    \text{finite derivative loss},
    \\[0.5ex]
    \frac{\G(\lambda,T)}{\log\lambda}\to +\infty
    &\quad\leadsto\quad
    \text{infinite derivative loss}.
\end{align*}

\paragraph{\textmd{\textit{The nondegenerate case}}}

The case where $b(t)\equiv 1$, namely the case of propagation speeds bounded from above and from below by positive constants, has been widely investigated in the literature. In this setting the main issue is not the degeneracy of the equation, but the possible oscillatory behavior of the coefficient, and different assumptions have been introduced in order to control such oscillations.

A first class of results concerns propagation speeds with a prescribed modulus of continuity. This point of view goes back to the approximated-energy method of Colombini, De Giorgi, and Spagnolo~\cite{1979-DGCS}. In modern terms, if $c$ has continuity modulus $\omega$, this method leads to an energy growth governed by the quantity
\begin{equation}
    \lambda\,\omega(1/\lambda).
    \nonumber
\end{equation}

A second class of results concerns coefficients which are regular in $(0,T]$, but whose derivatives may blow up as $t\to0^+$. More precisely, one assumes that
\[
    c\in C^0([0,T])\cap C^k((0,T])
\]
for some positive integer $k$, and measures the singular behavior of $c$ near the origin through suitable bounds on its derivatives. The first-order theory was developed, in this spirit, in~\cite{1980-Yagdjian,2002-SNS-ColDSaKin}. In the second-order theory the bound on $c''$ can be combined with the control of $c'$ in order to obtain logarithmic refinements of the derivative-loss estimates. This is the framework of~\cite{2003-BSM-ColDSaRei,2003-Hirosawa,2004-HirosawaReissig,2008-CicognaniHirosawaReissig}; see also~\cite{gg:CDSR-optimal}.

A third, more flexible, framework combines the previous two approaches, assuming at the same time a global continuity modulus for $c$ and estimates on its derivatives in $(0,T]$. This makes it possible to exploit both the global information provided by the modulus of continuity and the local information provided by derivative bounds \cite{2005-KinoshitaReissig,2007-DelSantoKinoshitaReissig}. A sharp formulation of this principle, in terms of an optimization problem which balances the two kinds of information, was obtained in~\cite{gg:OptDerLoss}.

Finally, we recall that the known thresholds in the nondegenerate theory are sharp in a strong sense: the pathological coefficients which realize the maximal derivative loss are not exceptional, but form residual sets in the appropriate functional classes~\cite{gg:residual}.

Our results can be viewed as a degenerate counterpart of this picture, where the usual assumptions on the oscillating coefficient are combined with a monotone factor that may vanish at the origin.

\paragraph{\textmd{\textit{The degenerate case}}}

The case where $c(t)\equiv 1$ and $b(0)=0$ was first investigated in~\cite{1983-CJS}. The result reveals a rather striking phenomenon. Without any monotonicity assumption on $b$, solutions may exhibit infinite derivative loss even when $b$ is of class $C^\infty$. The underlying mechanism can be described as follows. If, in a time interval $[t_1,t_2]$, the propagation speed oscillates between two values $b(t_1)$ and $b(t_2)$, then the energy of a solution may grow by a factor related to the ratio $b(t_2)/b(t_1)$. This ratio can be very large even when the absolute oscillation $b(t_2)-b(t_1)$ is small, provided that $b(t_1)$ is close to~$0$.

This heuristic suggests that monotonicity of the degenerate factor is not merely a technical assumption, but rather a structural condition which prevents this amplification mechanism. Degenerate coefficients with a prescribed structure, such as power-like degeneracies or more general time-dependent singularities and degeneracies, have been considered in several works; see for instance~\cite{1995-Tarama,1997-Yagdjian,Yamazaki1990}. This motivates the question of whether monotonicity alone is enough to control the derivative loss in the purely degenerate abstract equation.

\paragraph{\textmd{\textit{Our results}}}

In this paper we investigate how the possibly degenerate factor $b(t)$ interacts with the oscillations of $c(t)$ in determining the derivative loss of solutions. We begin with the purely degenerate case $c(t)\equiv 1$, considered in Theorem~\ref{thm:b-monotone}. In this case monotonicity of $b$ has a strong stabilizing effect on the derivative loss: regardless of how fast $b(t)$ vanishes at the origin, solutions lose at most one derivative. For special choices of $b(t)$, related phenomena were already considered in~\cite{1980-TaniguchiTozaki,1983-AlexandrianYagdjian}.

The situation changes in the first- and second-order theories, where $c$ is regular for positive times but may have a singular behavior at the origin. In these cases the derivative loss is described by minimization problems. In the first-order theory the minimum involves one frequency-dependent splitting point, while in the second-order theory it involves two such points. These minimization problems express the balance between the degeneracy of $b$ and the singular behavior of the derivatives of $c$. In particular, a sufficiently strong degeneracy of $b$ may compensate for singular oscillations of $c$ and reduce, or even eliminate, the derivative loss due to the oscillating factor. We refer to Theorems~\ref{thm:c0}, \ref{thm:c1}, and~\ref{thm:c2} for the precise statements.

\paragraph{\textmd{\textit{Overview of the technique}}}

The proofs are based on frequency-dependent energy estimates for the scalar equation \eqref{eqn:ODE}. The time interval is divided into regions where different energies are effective. Near the origin, where the propagation speed may be small, we use a Kowalevskian-type energy. Away from the degeneracy, hyperbolic energies become available. In the second-order theory a further Tarama-type energy is used in the region where the information on the second derivative of $c$ becomes relevant.

The transition points between these regions depend on the frequency $\lambda$. Optimizing their position leads precisely to the minimization problems that appear in the statements of the first- and second-order theories.

\paragraph{\textmd{\textit{Structure of the paper}}}

This paper is organized as follows. In Section~\ref{sec:statements} we state our main results, together with some applications to explicit propagation speeds. In Section~\ref{sec:energy} we prove the energy estimates that are the key tools in our analysis. These estimates involve three types of energies, which we call Kowalevskian, hyperbolic, and Tarama-type energies. Finally, in Section~\ref{sec:proofs} we combine the previous estimates in order to prove our main results.


\setcounter{equation}{0}
\section{Statements and examples}\label{sec:statements}

In this section we state the main results of this paper. We formulate them for solutions to the family of ordinary differential equations \eqref{eqn:ODE}, since the application to \eqref{eqn:basic} and \eqref{eqn:wave} is rather straightforward, as explained in the introduction. We start with the model case where the propagation speed contains only the monotone, possibly degenerate, factor, and no oscillatory term. Then we pass to the general case with oscillations.

\subsection{Degenerate case without oscillations}

We first consider the purely degenerate case
\begin{equation}
\ul''(t)+\lambda^{2}b(t)^{2}\ul(t)=0.
\label{eqn:main-ode}
\end{equation}
where $b$ is nondecreasing and nonnegative, but possibly vanishes at $t=0$, or even in a right neighborhood of $t=0$.

\begin{thm}[Degenerate monotone shape function]\label{thm:b-monotone}

Let $T$ be a positive real number, and let $b:[0,T]\to[0,+\infty)$ be a nondecreasing function. For every real number $\lambda\geq 0$ let us set
\begin{equation}
\tl:=\sup\{t\in[0,T]:\lambda tb(t)\leq 1\}.
\label{defn:tl}
\end{equation}

Then the following estimates hold true for every $\lambda\geq 0$.

\begin{enumerate}
\renewcommand{\labelenumi}{(\roman{enumi})}
\item \emph{(Estimate from above).} Every solution $\ul$ to (\ref{eqn:main-ode}) satisfies
\begin{equation}
\ul'(t)^{2}+\lambda^{2}\ul(t)^{2}\leq
3\exp(1)\cdot\left(1+b(T)^{2}\right)\cdot
\left[(1+\lambda^{2}t_{\lambda}^{2})\cdot\ul'(0)^{2}+\lambda^{2}\ul(0)^{2}\right]
\label{th:b(t)-above}
\end{equation}
for every $t\in[0,T]$.

\item \emph{(Optimality).} If $c_\lambda$ and $d_\lambda$ are two real numbers such that every solution $\ul$ to (\ref{eqn:main-ode}) satisfies the inequality
\begin{equation}
\ul'(t)^{2}+\lambda^{2}\ul(t)^{2}\leq
c_\lambda\cdot\ul'(0)^{2}+d_\lambda\cdot\lambda^{2}\ul(0)^{2}
\qquad
\forall t\in [0,T],
\label{hp:optimality}
\end{equation}
then necessarily
\begin{equation}
c_\lambda\geq \frac{1}{16}\lambda^{2}t_{\lambda}^{2}
\qquad
\forall\lambda\geq 0.
\label{th:b(t)-optimal}
\end{equation}

\end{enumerate}

\end{thm}

\begin{rmk}[On the mechanism and magnitude of derivative loss]
\begin{em}

We observe that, in this model case, the derivative loss of solutions depends only on the term $\lambda^2 \tl^2$ in estimate~(\ref{th:b(t)-above}), and hence, ultimately on $\tl$. We also observe that this term affects only the initial condition for the derivative. Since $\tl$ is clearly bounded from above by $T$, the maximal possible derivative loss is proportional to $\lambda^2$, which, for concrete wave-type equations, corresponds to the loss of one derivative. In addition, this derivative loss is \emph{instantaneous}: solutions immediately lose the required (even fractional) number of derivatives and subsequently remain in the same space.

\end{em}
\end{rmk}

The previous result is stated for the scalar equation \eqref{eqn:main-ode}, but it has an immediate consequence for the corresponding wave equation
\begin{equation}
u_{tt}(t,x)-b(t)^{2}\Delta u(t,x)=0.
\label{eqn:wave-b}
\end{equation}

Indeed, applying estimate~\eqref{th:b(t)-above} to each Fourier component gives the following statement, where the loss affects only the regularity of the initial velocity. For simplicity, we leave the space domain and the boundary conditions unspecified, since they play no role in the discussion below.

\begin{cor}[Wave equation with degenerate monotone propagation speed]\label{cor:b}

Let $T$ be a positive real number, and let $b:[0,T]\to[0,+\infty)$ be a nondecreasing function.

Then, for every real number $s\geq 0$, and for every pair of initial data in $H^{s+1}\times H^{s+1}$, the solution to \eqref{eqn:wave-b} lies in
\begin{equation}
C^{0}\left([0,T],H^{s+1}\right)\cap C^{1}\left([0,T],H^{s}\strut\right).
\label{space-u(t)}
\end{equation}

\end{cor}

In the sequel, we summarize the conclusion of Corollary~\ref{cor:b} by saying that solutions exhibit a derivative loss of the form
\begin{equation}
    H^{s+1}\times H^{s+1}
    \quad\leadsto\quad
    H^{s+1}\times H^{s}.
    \nonumber
\end{equation}
More refined conclusions can be obtained by looking at the actual size of $t_\lambda$. Indeed, estimate~\eqref{th:b(t)-above} shows that the relevant quantity is $\lambda^2t_\lambda^2$: the faster this quantity grows as $\lambda\to+\infty$, the larger the loss on the initial velocity.

The following examples illustrate this mechanism for some typical choices of the propagation speed.

\begin{ex}[Power-like propagation speed] \label{Examplepowerlike}
\begin{em}

Let us assume that $b(t)=t^{\ell}$ for some real number $\ell>0$. Then, for every $\lambda$ large enough, it turns out that $\tl=\lambda^{-1/(\ell+1)}$, and hence, \eqref{th:b(t)-above} reads as
\begin{equation}
\ul'(t)^{2}+\lambda^{2}\ul(t)^{2}\leq
3\exp(1)\left(1+T^{2\ell}\right)
\left[\left(1+\lambda^{2\ell/(\ell+1)}\right)\ul'(0)^{2}+\lambda^{2}\ul(0)^{2}\right].
    \nonumber
\end{equation}

In terms of wave equations, this means that solutions to~\eqref{eqn:wave-b} exhibit a derivative loss of the form
\begin{equation}
    H^{s+1}\times H^{s+\ell/(\ell+1)}
    \quad\leadsto\quad
    H^{s+1}\times H^{s},
    \nonumber
\end{equation}
and this amount of regularity on the initial velocity is optimal in view of the optimality statement in Theorem~\ref{thm:b-monotone}. As expected, the derivative loss $\ell/(\ell+1)$ increases with $\ell$, and approaches~0 as $\ell \to 0^+$, and~1 as $\ell \to +\infty$. In other words, stronger degeneracy of the propagation speed yields a larger derivative loss.

For related results on weakly hyperbolic equations with time-dependent singularity or degeneracy, including power-like models, we refer to~\cite{Yamazaki1990}.

\end{em}
\end{ex}

\begin{ex}[Propagation speeds with borderline vanishing orders]
\begin{em}

Let us assume that $b(t)\simeq|\log t|^{-1}$ as $t\to 0^+$. Then one can check that $\lambda^2 \tl^2\simeq\log^2\lambda$, which means that solutions exhibit an infinitesimally small derivative loss. In terms of the wave equation, this means that
\begin{equation}
    H^{s+1}\times (\log \langle D \rangle)^{-1} H^{s}
    \quad\leadsto\quad
    H^{s+1}\times H^{s}.
    \nonumber
\end{equation}

At the opposite end of the spectrum of possible choices, let us consider the case where $b(t)\simeq\exp(-1/t)$ as $t\to 0^+$. Then one can check that $\lambda^2 \tl^2\simeq\lambda^2(\log\lambda)^{-2}$, which means that the derivative loss is maximal, namely~1, up to a logarithmic factor. In terms of the wave equation, this means that
\begin{equation}
    H^{s+1}\times (\log \langle D \rangle) H^{s+1}
    \quad\leadsto\quad
    H^{s+1}\times H^{s}.
    \nonumber
\end{equation}

\end{em}
\end{ex}

\begin{ex}[Totally degenerate propagation speed]
\begin{em}

Let us consider the limiting case where $b(t)\equiv 0$ for every $t\in[0,T]$. Then, for every $\lambda \geq 0$, it turns out that $\tl = T$, and hence~\eqref{th:b(t)-above} reads as
\begin{equation}
\ul'(t)^{2}+\lambda^{2}\ul(t)^{2}\leq
3\exp(1)\cdot\left[(1+\lambda^{2}T^{2})\cdot\ul'(0)^{2}+\lambda^{2}\ul(0)^{2}\right].
    \nonumber
\end{equation}

The derivative loss associated with the term $\lambda^{2}T^{2}$ is the maximal possible one, and agrees with what is obtained from the explicit solution, namely $\ul(t)=\ul(0)+\ul'(0)\cdot t$.

\end{em}
\end{ex}


\subsection{Degenerate cases with controlled oscillations}

We now pass to the general case where the propagation speed has the form given in~(\ref{eqn:gamma=b*c}). Thus we consider the family of ordinary differential equations
\begin{equation}
\ul''(t)+\lambda^{2}b(t)^{2}c(t)^2\ul(t)=0.
\label{eqn:main-ode-c}
\end{equation}
As before, $b$ is monotone and possibly degenerate, while $c$ is bounded from above and from below by positive constants, but possibly oscillates.

\subsubsection{Zeroth-order control on the whole interval}

In the first result, we control the oscillations of $c$ on the whole interval by measuring its oscillation at each scale. When $c$ is continuous, this is the usual modulus-of-continuity framework; in general, however, no continuity assumption is needed.

\begin{thm}[$C^0$-theory] \label{thm:c0}

Let $T$, $\nu_1$, $\nu_2$ be three positive real numbers with $\nu_1\leq\nu_2$. Let $b:[0,T]\to[0,+\infty)$ be a nondecreasing function, and let $c:[0,T]\to\re$ be a measurable function that satisfies the strict hyperbolicity condition
\begin{equation}
    \nu_1\leq c(t)\leq\nu_2
    \qquad
    \forall t\in[0,T].
    \label{hp:c-sh}
\end{equation}

For every real number $s>0$ we consider the oscillation $\omega(s)$ of $c$ at scale $s$, defined as
\begin{equation}
    \omega(s):=\sup\left\{|c(t_2)-c(t_1)|:(t_1,t_2)\in[0,T]^2,\ |t_2-t_1|\leq s\right\}.
    \label{defn:c-omega}
\end{equation}

For every real number $\lambda> 0$, we define the numbers
\begin{equation}
\tl:=\sup\{t\in[0,T]:\nu_2\lambda tb(t)\leq 1\},
\label{defn:tl-nu2}
\end{equation}
and
\begin{equation}
    \Phi_0(\lambda,T):=\lambda\,\omega\left(\frac{1}{\lambda}\right)\cdot
    \left(\frac{2T}{\nu_1}+\frac{2\nu_2}{\nu_1}\int_0^T b(t)\,dt\right).
    \label{defn:Phi0}
\end{equation}

Then every solution $\ul$ to (\ref{eqn:main-ode-c}) satisfies
\begin{eqnarray}
    \ul'(t)^2+\lambda^2\ul(t)^2 & \leq &
    K_0(\nu_1,\nu_2)\cdot
    \exp(\Phi_0(\lambda,T))\cdot\left(1+\nu_2^2 b(T)^2\right)
    \nonumber
    \\[1ex]
    & &
    \cdot\left[(1+\lambda^2\tl^2)\cdot\ul'(0)^2+\lambda^2\ul(0)^2\right]
    \label{th:C0-theory}
\end{eqnarray}
for every $t\in [0,T]$, where
\begin{equation}
    K_0(\nu_1,\nu_2):=
    3\left(\frac{\nu_2}{\nu_1}\right)^2\exp\left(\frac{\nu_2}{\nu_1}\right).
    \label{defn:K0}
\end{equation}

\end{thm}

Let us comment on the assumptions and conclusions of Theorem~\ref{thm:c0}.

\begin{rmk}[On the oscillation function]
\begin{em}

The function $\omega$ appearing in Theorem~\ref{thm:c0} is defined directly from $c$ as its oscillation at scale~$s$. Thus no additional assumption on $c$ is hidden in the notation. When $c$ is continuous, then $\omega$ is the usual modulus of continuity. In general, however, $\omega(s)$ need not tend to~$0$ as $s\to0^+$.

At the opposite extreme, one always has
\[
    \omega(s)\leq \nu_2-\nu_1
    \qquad
    \forall s>0,
\]
by the strict hyperbolicity condition~(\ref{hp:c-sh}). Therefore, the result applies also to merely measurable propagation speeds. In this case the exponent $\Phi_0(\lambda,T)$ is at most proportional to~$\lambda$, in agreement with the classical theory, where one has to require analytic regularity of the initial data in the absence of any effective control on the oscillations.

\end{em}
\end{rmk}

\begin{rmk}[On the derivative loss]
\begin{em}

Let us comment on the different terms that appear in~(\ref{th:C0-theory}).

\begin{itemize}

\item If the function $c$ is constant, then $\omega(s)\equiv0$ and $\nu_1=\nu_2$, and~(\ref{th:C0-theory}) reduces to the estimate of Theorem~\ref{thm:b-monotone}.

\item The factors $K_0(\nu_1,\nu_2)$ and $1+\nu_2^2 b(T)^2$ do not depend on $\lambda$, and hence, do not contribute to the derivative loss.

\item The derivative loss is the sum of two contributions. The contribution due to the degeneracy of $b(t)$ is represented by the term $\lambda^2\tl^2$ in front of the initial velocity. The contribution due to the oscillations of $c(t)$ is represented by the exponential factor $\exp(\Phi_0(\lambda,T))$, whose exponent is proportional to $\lambda\,\omega(1/\lambda)$.

\item In this zeroth-order theory there is no genuine interaction between $b(t)$ and $c(t)$. Their contributions simply multiply at the level of the estimate, and therefore, add up at the level of derivative loss. Both contributions are separately known to be optimal.

\item In contrast with the contribution of $b(t)$, the derivative loss due to $c(t)$ is progressive in time. Indeed, in~(\ref{defn:Phi0}) the coefficient of $\lambda\,\omega(1/\lambda)$ increases with $T$, consistently with the fact that the possible oscillations of $c$ may accumulate their effect along the whole interval.

\item The derivative loss due to $b(t)$ affects only the regularity required on the initial velocity $\ul'(0)$, while the derivative loss due to $c(t)$ affects both initial data $\ul'(0)$ and $\ul(0)$.

\end{itemize}

\end{em}
\end{rmk}

\begin{rmk}[Finite derivative loss]
\begin{em}

If we are interested in solutions with at most finite derivative loss, then
$\Phi_0(\lambda,T)$ has to be bounded by a multiple of $\log\lambda$ as
$\lambda\to+\infty$. Since the only $\lambda$-dependent term in
$\Phi_0(\lambda,T)$ is $\lambda\,\omega(1/\lambda)$, this requires
\[
    \omega(\sigma)\leq C\sigma|\log\sigma|
\]
for $\sigma$ small enough. In
particular, $\omega(\sigma)\to0$ as $\sigma\to0^+$, and therefore $c$ is
continuous.

Thus, in the $C^0$-theory, genuinely discontinuous coefficients are
allowed, but they lead in general to infinite derivative loss. Finite derivative
loss corresponds precisely to the classical log-Lipschitz scale, which is the
same threshold as in the nondegenerate theory.

\end{em}
\end{rmk}


\subsubsection{First-order control near the origin}

We now consider an oscillating factor $c(t)$ which is of class $C^1$ for positive times, but may have a singular behavior as $t\to0^+$. The behavior near the origin is measured by the oscillation of $c$ in the interval $(0,s]$, while the behavior away from the origin is measured by the total variation of $c$ in $(s,T)$. The splitting point $s$ is then chosen in a frequency-dependent way.

\begin{thm}[$C^1$-theory] \label{thm:c1}

Let $T$, $\nu_1$, $\nu_2$ be three positive real numbers with $\nu_1\leq\nu_2$. Let $b:[0,T]\to[0,+\infty)$ be a nondecreasing function, and let $c\in C^1((0,T])$ be a function satisfying the strict hyperbolicity condition
\begin{equation}
    \nu_1\leq c(t)\leq\nu_2
    \qquad
    \forall t\in(0,T].
    \label{hp:c1-c-sh}
\end{equation}

For every $s\in(0,T]$ we consider the oscillation of $c$ in $(0,s]$, defined as
\begin{equation}
    \osc(c,(0,s]):=
    \sup\left\{|c(t_2)-c(t_1)|:(t_1,t_2)\in(0,s]^2\right\}.
    \label{defn:osc-0s}
\end{equation}

For every real number $\lambda\geq 0$, we define $\tl$ as in~(\ref{defn:tl-nu2}), the function $\psi_{1,\lambda}:[\tl,T]\to\re$ by
\begin{equation}
    \psi_{1,\lambda}(s):=
    \frac{2\nu_2}{\nu_1}\lambda\cdot\osc(c,(0,s])\int_{\tl}^{s}b(t)\,dt+
    \frac{2}{\nu_1}\int_s^T|c'(t)|\,dt,
    \label{defn:psi-1}
\end{equation}
and the number
\begin{equation}
    \Phi_1(\lambda,T):=\min\left\{\psi_{1,\lambda}(s):s\in[\tl,T]\right\}.
    \label{defn:Phi1}
\end{equation}

Then every solution $\ul$ to~(\ref{eqn:main-ode-c}) satisfies
\begin{eqnarray}
    \ul'(t)^2+\lambda^2\ul(t)^2 & \leq &
    K_0(\nu_1,\nu_2)\cdot\exp(\Phi_1(\lambda,T))\cdot\left(1+\nu_2^2b(T)^2\right)
    \nonumber
    \\[0.5ex]
    & &
    \cdot\left[(1+\lambda^2\tl^2)\cdot\ul'(0)^2+\lambda^2\ul(0)^2\right]
    \label{th:C1-theory}
\end{eqnarray}
for every $t\in [0,T]$, where $K_0(\nu_1,\nu_2)$ is defined by~(\ref{defn:K0}).

\end{thm}

Let us comment on the assumptions and conclusions of Theorem~\ref{thm:c1}.

\begin{rmk}[Local oscillation vs global modulus-type control]\label{rmk:c1-omega}
\begin{em}

In Theorem~\ref{thm:c1}, the behavior of $c$ near the origin is measured by the actual local oscillation $\osc(c,(0,s])$. If one assumes instead a global modulus-type control on the whole interval, then the same argument gives an analogous estimate in which the first term in~(\ref{defn:psi-1}) contains $\omega(1/\lambda)$ in place of $\osc(c,(0,s])$. This is closer to the classical nondegenerate formulation, but may be less adapted to the actual behavior of $c$ near the origin.

In this sense, the formulation of Theorem~\ref{thm:c1} is more intrinsic: it does not require choosing an a priori modulus of continuity, and it exploits only the oscillation that is really seen before the splitting point.

\end{em}
\end{rmk}

\begin{rmk}[Genesis of the derivative loss]\label{rmk:c1-der-loss}
\begin{em}

As in the $C^0$ theory, the contributions of $b(t)$ and $c(t)$ to the derivative loss are represented by the terms $\lambda^2 \tl^2$ and $\exp\bigl(\Phi_1(\lambda,T)\bigr)$, respectively. The new feature is that the contribution of $c(t)$ is no longer determined by a single explicit quantity, but by the minimization problem in~(\ref{defn:Phi1}).

This minimization reflects a compromise between two opposite requirements. On the one hand, the dangerous factor $\lambda$ in the first term of $\psi_{1,\lambda}$ can be compensated for by choosing $s$ close to $\tl$, where the oscillation of $c$ in $(0,s]$ is measured only near the origin and the integral of $b$ from $\tl$ to $s$ is small. On the other hand, if $s$ is too small, the integral of $|c'|$ in the second term of $\psi_{1,\lambda}$ may become too large. The quantity $\Phi_1(\lambda,T)$ is obtained by choosing the best frequency-dependent splitting point $\ssl\in[\tl,T]$ between these two regimes.

This is the first place where the degenerate factor $b$ and the oscillating factor $c$ genuinely interact: the optimal choice of $s$ depends at the same time on the growth of $b$, on the oscillation of $c$ near the origin, and on the singular behavior of $c'$ away from the origin.

\end{em}
\end{rmk}

The next corollary identifies a natural scale for the balance between the oscillations of $c$ and the growth of $|c'|$. It shows that, if $|c'|$ is controlled by the scale $b/B$, where $B$ is a primitive of $b$, then the contribution of the oscillating factor is at most logarithmic in the frequency, even without any smallness assumption on the oscillation of $c$ near the origin.

\begin{cor}[A natural scale]\label{cor:natural-scale-c1}

Let $b$ and $c$ be as in Theorem~\ref{thm:c1}, and let us assume in addition that $b(t)>0$ for every $t\in(0,T]$. We set
\begin{equation}
    B(t):=\int_0^t b(\tau)\,d\tau
    \qquad
    \forall t\in[0,T],
    \label{defn:B}
\end{equation}
and we suppose that there exists a nonnegative real number $M$ such that
\begin{equation}
    |c'(t)|\leq M\frac{b(t)}{B(t)}
    \qquad
    \forall t\in(0,T].
    \nonumber
\end{equation}

    Then one has
    \begin{equation}
        \Phi_1(\lambda,T)\leq
        \frac{2M}{\nu_1}\log\lambda+
        \frac{2\nu_2}{\nu_1}(\nu_2-\nu_1)B(T)
        \qquad
        \forall\lambda\geq 1.
        \label{th:cor-c1}
    \end{equation}

\end{cor}

The following examples show how this natural scale specializes to familiar classes of coefficients $b$. In both examples no smallness assumption is made on the oscillation of $c$ near the origin.

\begin{ex}[Power-like degenerate factor]
\begin{em}

Let us assume that $b(t)=t^\ell$ for some $\ell\geq 0$ (this includes the nondegenerate case as well), so that
\begin{equation}
    B(t)=\int_0^t \tau^\ell\,d\tau
    =
    \frac{t^{\ell+1}}{\ell+1}
    \qquad\quad\text{and}\quad\qquad
    \frac{b(t)}{B(t)}
    =
    \frac{\ell+1}{t}.
    \nonumber
\end{equation}
Thus Corollary~\ref{cor:natural-scale-c1} shows that, if
\begin{equation}
    |c'(t)|\leq \frac{M}{t}
    \qquad
    \forall t\in(0,T],
    \nonumber
\end{equation}
then the contribution of the oscillating factor $c(t)$ produces at most finite derivative loss.

In other words, for power-like degeneracies the borderline growth of $c'$ is the classical one, namely of order $1/t$. In this case the degeneracy of $b$ does not enlarge the class of admissible singularities for $c'$ beyond the classical threshold.

\end{em}
\end{ex}

\begin{ex}[Exponentially flat degenerate factor]
\begin{em}

Let us assume that, for some $\alpha>0$,
\begin{equation}
    b(t)=\exp\left(-\frac{1}{t^\alpha}\right)
    \qquad
    \text{for }t\text{ close to }0.
    \nonumber
\end{equation}
Then one has
\begin{equation}
    B(t):=\int_0^t b(\tau)\,d\tau
    \simeq
    t^{\alpha+1}\exp\left(-\frac{1}{t^\alpha}\right)
    \qquad
    \text{as $t\to0^+$},
    \nonumber
\end{equation}
and therefore,
\begin{equation}
    \frac{b(t)}{B(t)}
    \simeq
    \frac{1}{t^{\alpha+1}}.
    \nonumber
\end{equation}
Thus Corollary~\ref{cor:natural-scale-c1} shows that, if
\begin{equation}
    |c'(t)|\leq \frac{M}{t^{\alpha+1}}
    \qquad
    \forall t\in(0,T],
    \nonumber
\end{equation}
then the contribution of the oscillating factor $c(t)$ produces at most finite derivative loss.

This example shows that a flatter degeneracy of the factor $b$ can compensate for a stronger singularity of $c'$ at the origin. In the present case, the admissible growth is of order $t^{-\alpha-1}$, which is larger than the classical borderline $t^{-1}$.

We observe that in this case the resulting propagation speed $\gamma(t)=b(t)c(t)$ is of class $C^1$ up to the origin, with $\gamma'(0)=0$. In the nondegenerate theory, this regularity would already prevent any derivative loss. Here, however, the propagation speed vanishes at the origin, and therefore the problem remains genuinely degenerate. The point of the example is that the flatness of $b$ compensates for the singular behavior of $c'$, without reducing the situation to a trivial nondegenerate one.

\end{em}
\end{ex}

\subsubsection{Second-order control away from the origin}

In the final result, the coefficient $c$ is assumed to be of class $C^2$ for positive times, but may have a singular behavior as $t\to0^+$. As in the first-order theory, the behavior near the origin is measured by the oscillation of $c$ in the interval $(0,s]$. Away from the origin, however, we also exploit the second-order information on $c$. This leads to a minimization problem with two frequency-dependent splitting points.

\begin{thm}[$C^2$-theory]\label{thm:c2}

Let $T$, $\nu_1$, $\nu_2$ be three positive real numbers with $\nu_1\leq\nu_2$. Let $b:[0,T]\to[0,+\infty)$ be a nondecreasing function, and let $c\in C^2((0,T])$ be a function satisfying the strict hyperbolicity condition~(\ref{hp:c1-c-sh}).

For every $s\in(0,T]$ we consider the oscillation of $c$ in $(0,s]$, defined as in~(\ref{defn:osc-0s}).

For every $\lambda>0$, we define $\tl$ as in~(\ref{defn:tl-nu2}), and we define $A_\lambda$ as follows. If
\begin{equation}
    |c'(t)|\leq \lambda b(t)c(t)^2
    \qquad
    \forall t\in[\tl,T],
    \nonumber
\end{equation}
then we set $A_\lambda:=\tl$. Otherwise, we set
\begin{equation}
    A_\lambda:=
    \sup\left\{t\in[\tl,T]:
    |c'(t)|>\lambda b(t)c(t)^2
    \right\}.
    \label{defn:Al}
\end{equation}

Finally, we consider the function $\psi_{2,\lambda}:[\tl,T]^2\to\re$ defined by
\begin{eqnarray}
    \psi_{2,\lambda}(s,a) & := &
    \frac{2\nu_2}{\nu_1}\lambda\cdot\osc(c,(0,s])\int_{\tl}^{s}b(t)\,dt+
    \frac{2}{\nu_1}\int_s^a |c'(t)|\,dt
    \nonumber
    \\[1ex]
    & &
    \mbox{}+\frac{3}{\lambda}
    \int_{a}^{T}\frac{1}{b(t)}
    \left(\frac{|c''(t)|}{c(t)^2}+\frac{c'(t)^2}{c(t)^3}\right)\,dt,
    \label{defn:psi-2}
\end{eqnarray}
and the number
\begin{equation}
    \Phi_2(\lambda,T):=\min\left\{\psi_{2,\lambda}(s,a):
    A_\lambda\leq a\leq T,\ \tl\leq s\leq a\right\}.
    \label{defn:Phi2}
\end{equation}

Then every solution $\ul$ to~(\ref{eqn:main-ode-c}) satisfies
\begin{eqnarray}
    \ul'(t)^2+\lambda^2\ul(t)^2 & \leq &
    K_2(\nu_1,\nu_2)\cdot\exp\left(\Phi_2(\lambda,T)\right)\cdot\left(1+\nu_2^2 b(T)^2\right)
    \nonumber
    \\[0.5ex]
    & &
    \cdot\left[(1+\lambda^2\tl^2)\cdot\ul'(0)^2+\lambda^2\ul(0)^2\right]
    \label{th:C2-theory}
\end{eqnarray}
for every $t\in [0,T]$, where
\begin{equation}
    K_{2}(\nu_1,\nu_2):=
    15\left(\frac{\nu_2}{\nu_1}\right)^3\exp\left(\frac{\nu_2}{\nu_1}\right).
    \nonumber
\end{equation}

\end{thm}

Let us comment on the assumptions and conclusions of Theorem~\ref{thm:c2}.

\begin{rmk}[Comparison with the first-order theory]
\begin{em}

This result can be regarded as an extension of Theorem~\ref{thm:c1}, and the comments in Remarks~\ref{rmk:c1-omega} and~\ref{rmk:c1-der-loss} remain valid also in this setting.

The function $\psi_{2,\lambda}$ is analogous to $\psi_{1,\lambda}$, but contains one additional term. The three terms in $\psi_{2,\lambda}$ have different homogeneities with respect to $\lambda$. The first term has degree~1 and involves the zeroth-order behavior of $c(t)$ near the origin, measured by the oscillation of $c$ in $(0,s]$. The second term has degree~0 and involves the first derivative of $c(t)$. The third term has degree~$-1$ and involves the second derivative of $c(t)$.

Thus the minimization problem in~(\ref{defn:Phi2}) contains two splitting points, $\ssl$ and $\al$, with precise roles. The first one separates the region where one uses only the oscillation of $c$ near the origin from the region where one exploits the first-order information on $c$. The second one separates the latter region from the region where the second-order information on $c$ becomes effective.

\end{em}
\end{rmk}

The next corollary identifies a natural scale in the second-order theory. It shows that, even without any smallness assumption on the oscillation of $c$ near the origin, the simultaneous control of $c'$ and $c''$ may be strong enough to keep the contribution of the oscillating factor logarithmic in the frequency, and hence, to prevent any additional derivative loss due to oscillations.

\begin{cor}[A natural scale for the second-order theory]\label{cor:natural-scale-c2}

Let $b$ and $c$ be as in Theorem~\ref{thm:c2}, and let us assume in addition that $b(t)>0$ for every $t\in(0,T]$. We define $B(t)$ as in (\ref{defn:B}), and we suppose that $B(T)\leq 1$. In addition, we assume that there exists a positive real number $M$ such that, for every $t\in(0,T]$,
\begin{equation}
    |c'(t)|\leq M\frac{b(t)}{B(t)}|\log B(t)|
    \qquad\text{and}\qquad
    |c''(t)|\leq
    M^2\frac{b(t)^2}{B(t)^2}|\log B(t)|^2.
    \label{hp:c2-second}
\end{equation}

Then, for every $\lambda$ large enough to satisfy
\begin{equation}
    \lambda>\frac{M}{\nu_1^2}\frac{|\log B(T)|}{B(T)},
    \qquad\qquad
    \log(\nu_2\lambda)>\frac{\nu_1^2}{M\nu_2},
    \label{hp:lambda-c2}
\end{equation}
one has
\begin{equation}
    \Phi_2(\lambda,T)\leq
    \left[
    \frac{2M\nu_2(\nu_2-\nu_1)}{\nu_1^3}
    +
    \frac{6M(\nu_1+1)}{\nu_1}
    \right]\log(\nu_2\lambda)
    +
    \frac{9\nu_2M^2(\nu_1+1)}{\nu_1^3}.
    \nonumber
\end{equation}

\end{cor}

\begin{rmk}[On the two splitting points]
\begin{em}

In the proof of Corollary~\ref{cor:natural-scale-c2}, the specific choice of the second-order scale in~(\ref{hp:c2-second}) allows us to choose an admissible competitor in the minimization problem for $\Phi_2(\lambda,T)$ with $s=a$. Thus the two splitting points collapse into a single one, and the interval is effectively divided into only two regions.

This simplification is a consequence of the special balance between the first- and second-order assumptions in~(\ref{hp:c2-second}). In the general second-order theory, however, there is no reason for the two optimal splitting points to coincide. When they are distinct, the minimization problem naturally involves three regions, corresponding to the three different pieces of information which enter the estimate.

This phenomenon already appears in the nondegenerate theory, in the proof of~\cite[Theorem~2.1]{gg:CDSR-optimal}, where two independent splitting points and hence, three regions enter the estimate. A closely related three-region construction is discussed in~\cite[Example~6.10]{gg:summa-theologica}.

\end{em}
\end{rmk}

\begin{rmk}[On the absence of higher-order theories]
\begin{em}

One might wonder whether the first- and second-order results presented above are the first two steps of a hierarchy of higher-order theories. In the present framework this is not the case. In fact, already in the nondegenerate setting, the information relevant to derivative-loss estimates saturates at the second order.

This is due to structural reasons rather than to a limitation of the method. A detailed discussion of this point, including the nondegenerate case and the role of higher-order information, can be found in~\cite{gg:summa-theologica}.

\end{em}
\end{rmk}


\setcounter{equation}{0}
\section{General energy estimates}\label{sec:energy}

In this section we prove the energy estimates that will be used in the proof of the main results. We work under the regularity assumptions on the coefficients that are needed in order to differentiate the corresponding energies. The passage from regular monotone coefficients to arbitrary monotone coefficients will be carried out later, in Section~\ref{sec:proofs}, at the level of the main results.

The estimates correspond to the three regimes used throughout the paper: the Kowalevskian regime close to the degeneracy, the hyperbolic regime based on first-order information, and the Tarama-type regime based on second-order information.


The first estimate is based on a Kowalevskian energy, namely a linear combination of $\ul'(t)^2$ and $\lambda^2\ul(t)^2$ with positive constant coefficients. This estimate is useful only in the initial regime, more precisely, up to the threshold time $\tl$.

\begin{prop}[General estimate of Kowalevskian type]\label{prop:kow}

Let $T$, $d_0$, $\nu_1$, and $\nu_2$ be positive real numbers with $\nu_1\leq d_0\leq \nu_2$. Let $b:[0,T]\to[0,+\infty)$ be a nondecreasing continuous function, and let $c:(0,T)\to[\nu_1,\nu_2]$ be a continuous function.

For every real number $\lambda\geq 0$, we define $\tl$ as in~\eqref{defn:tl-nu2}.

Then every solution to~\eqref{eqn:main-ode-c} satisfies
\begin{equation}
    \ul'(t)^2+\lambda^2 b(\tl)^2d_0^2\,\ul(t)^2\leq
    \exp\left(\frac{\nu_2}{\nu_1}\right)
    \left[\ul'(0)^2+\lambda^2 b(\tl)^2d_0^2\,\ul(0)^2\right]
    \label{th:est-kow}
\end{equation}
for every $t\in[0,\tl]$.

\end{prop}

\begin{proof}

Since $b$ is bounded from above, from~\eqref{defn:tl-nu2} it follows that $\tl>0$, and by continuity also
\begin{equation}
\nu_2\lambda\tl b(\tl)\leq 1.
\label{est:b-tl}
\end{equation}

In the interval $[0,\tl]$ we consider the energy
\begin{equation}
\Ek(t):=\ul'(t)^{2}+\lambda^{2}b(\tl)^{2}d_0^2\ul(t)^{2}.
    \nonumber
\end{equation}
Due to the regularity of $b$ and $c$, this energy is continuous in $[0,\tl]$, and of class $C^{1}$ in $(0,\tl)$, and its time derivative is equal to
\begin{equation}
\Ek'(t)=2\lambda^{2}\ul'(t)\ul(t)
\left(b(\tl)^{2}d_0^2-b(t)^{2}c(t)^2\right)
\qquad
\forall t\in(0,\tl).
    \nonumber
\end{equation}

Since $b$ is nondecreasing, and both $d_0$ and $c(t)$ lie in the interval $[\nu_1,\nu_2]$, it follows that
\begin{equation}
    \left|b(\tl)^{2}d_0^2-b(t)^{2}c(t)^2\right|\leq
    \max\left\{b(\tl)^{2}d_0^2,b(t)^{2}c(t)^2\right\}\leq
    b(\tl)^2 \nu_2^2
    \nonumber
\end{equation}
for every $t\in[0,\tl]$, and therefore
\begin{equation}
\Ek'(t)\leq
2\lambda^{2}|\ul'(t)|\cdot|\ul(t)|\cdot b(\tl)^{2}\nu_2^2
    \nonumber
\end{equation}
in the same interval. Since
\[
    2\lambda b(\tl)d_0\,|\ul'(t)|\cdot|\ul(t)|
    \leq \Ek(t),
\]
we deduce that
\begin{equation}
\Ek'(t)\leq
\lambda b(\tl)\frac{\nu_2^2}{d_0}\cdot\Ek(t)
\leq
\lambda b(\tl)\frac{\nu_2^2}{\nu_1}\cdot\Ek(t)
\qquad
\forall t\in(0,\tl).
    \nonumber
\end{equation}

Integrating this differential inequality, and recalling~\eqref{est:b-tl}, we conclude that
\begin{equation}
    \Ek(t)\leq
    \exp\left(\frac{\nu_2^2}{\nu_1}\lambda b(\tl)\tl\right)\Ek(0)
    \leq
    \exp\left(\frac{\nu_2}{\nu_1}\right)\Ek(0)
    \nonumber
\end{equation}
for every $t\in[0,\tl]$, which proves~\eqref{th:est-kow}.
\end{proof}


The second energy estimate is based on a hyperbolic energy, namely a linear combination of $\ul'(t)^2$ and $\lambda^2\ul(t)^2$ in which the coefficient of the second term depends on time. In contrast with the Kowalevskian energy, this energy remains effective on larger time intervals.

\begin{prop}[General estimate of hyperbolic type]\label{prop:hyp}

Let $[\alpha,\beta]\subseteq\re$ be an interval, let $\nu_1$ and $\nu_2$ be two positive real numbers with $\nu_1\leq\nu_2$, and let $\lambda\geq 0$ be a real number.

Let $b:[\alpha,\beta]\to(0,+\infty)$ be a nondecreasing function of class $C^1$, let $c:(\alpha,\beta)\to[\nu_1,\nu_2]$ be a continuous function, and let $d:[\alpha,\beta]\to[\nu_1,\nu_2]$ be a function of class $C^1$.

Then every solution to~\eqref{eqn:main-ode-c} defined on $[\alpha,\beta]$ satisfies
\begin{eqnarray}
    \ul'(t)^2+\lambda^2 b(t)^2 d(t)^2\ul(t)^2 & \leq &
    \exp(\Phih(\lambda,\alpha,\beta))\cdot\frac{b(t)^2}{b(\alpha)^2}
    \nonumber
    \\
    & &
    \cdot\left[\ul'(\alpha)^2+\lambda^2 b(\alpha)^2d(\alpha)^2\ul(\alpha)^2\right]
    \label{th:est-hyp}
\end{eqnarray}
for every $t\in[\alpha,\beta]$, where
\begin{equation}
\Phih(\lambda,\alpha,\beta):=
\frac{2}{\nu_1}\int_{\alpha}^{\beta}|d'(t)|\,dt+
\frac{2\nu_2}{\nu_1}\lambda\cdot
\sup_{t\in[\alpha,\beta]}|d(t)-c(t)|\cdot
\int_{\alpha}^\beta b(t)\,dt.
\label{defn:Phih}
\end{equation}

\end{prop}

\begin{proof}

In the interval $[\alpha,\beta]$ we consider the energy
\begin{equation}
\Eh(t):=\ul'(t)^{2}+\lambda^{2}b(t)^{2}d(t)^2\ul(t)^{2}.
    \nonumber
\end{equation}
Due to the time regularity of $b$, $c$, and $d$, this energy is continuous in $[\alpha,\beta]$ and of class $C^{1}$ in $(\alpha,\beta)$, where its time derivative is equal to
\begin{eqnarray*}
\Eh'(t)& = & 2\lambda^{2}b(t)b'(t)d(t)^2\ul(t)^{2}+
2\lambda^2 b(t)^2 d(t)d'(t)\ul(t)^2
\\
& &
\mbox{}+2\lambda^2 b(t)^2\ul(t)\ul'(t)\left(d(t)^2-c(t)^2\right)
\\
& =: &
L_1(t)+L_2(t)+L_3(t).
\end{eqnarray*}

Let us estimate the three terms. Since $b(t)>0$ and $b'(t)\geq 0$, for the first term we obtain
\begin{equation}
    L_1(t)\leq \frac{2b'(t)}{b(t)}\cdot\Eh(t).
    \nonumber
\end{equation}
Similarly, for the second term we obtain
\begin{equation}
    L_2(t)\leq
    \frac{2|d'(t)|}{d(t)}\cdot\Eh(t)\leq
    \frac{2}{\nu_1}|d'(t)|\cdot\Eh(t).
    \nonumber
\end{equation}

Finally, since both $c(t)$ and $d(t)$ lie in $[\nu_1,\nu_2]$, we have
\begin{equation}
    \left|d(t)^2-c(t)^2\right|=
    \left|d(t)+c(t)\right|\cdot\left|d(t)-c(t)\right|\leq
    2\nu_2\cdot\sup_{\tau\in[\alpha,\beta]}|d(\tau)-c(\tau)|.
    \nonumber
\end{equation}
Using also the inequality
\begin{equation}
    2\lambda b(t)d(t)|\ul(t)|\cdot|\ul'(t)|\leq \Eh(t),
    \nonumber
\end{equation}
we obtain
\begin{eqnarray}
    L_3(t)
    & \leq &
    \frac{2\nu_2}{\nu_1}\lambda\cdot
    \sup_{\tau\in[\alpha,\beta]}|d(\tau)-c(\tau)|\cdot b(t)\cdot\Eh(t).
    \nonumber
\end{eqnarray}

From all these estimates we deduce that
\begin{equation*}
\Eh'(t)\leq
\frac{2b'(t)}{b(t)}\Eh(t)+
\left(\frac{2}{\nu_1}|d'(t)|+
\frac{2\nu_2}{\nu_1}\lambda\cdot\sup_{\tau\in[\alpha,\beta]}|d(\tau)-c(\tau)|\cdot b(t)\right)\Eh(t)
\end{equation*}
for every $t\in(\alpha,\beta)$. Integrating this differential inequality in $[\alpha,t]$, and using
\begin{equation}
    \exp\left(2\int_\alpha^t\frac{b'(\tau)}{b(\tau)}\,d\tau\right)
    =
    \frac{b(t)^2}{b(\alpha)^2},
    \label{eqn:int-b^2}
\end{equation}
we obtain~\eqref{th:est-hyp}.
\end{proof}


The third estimate is based on a more refined energy, namely a full quadratic form in $\ul'(t)$ and $\ul(t)$ with time-dependent coefficients, similar to that introduced in~\cite{2007-EJDE-Tarama}. The coefficients are chosen so that, when differentiating the energy, all leading terms cancel out. The remaining terms are of order $\lambda^{-1}$, and therefore, become effective in the high-frequency regime.

The estimate requires a smallness condition on $c'$. In the applications to the main results, this condition will be imposed only in the region where the Tarama-type energy is used, and the corresponding interval will be chosen in a frequency-dependent way.

\begin{prop}[General estimate of Tarama type]\label{prop:tarama}

Let $[\alpha,\beta]\subseteq\re$ be an interval, and let $\lambda>0$ be a real number.

Let $b:[\alpha,\beta]\to(0,+\infty)$ be a nondecreasing function of class $C^1$, and let $c:[\alpha,\beta]\to(0,+\infty)$ be a function of class $C^2$ such that
\begin{equation}
    |c'(t)|\leq
    \lambda b(t)c(t)^2
    \qquad
    \forall t\in[\alpha,\beta].
    \label{hp:tar-c'}
\end{equation}

Then every solution to~\eqref{eqn:main-ode-c} defined on $[\alpha,\beta]$ satisfies
\begin{eqnarray}
    \ul'(t)^2+\lambda^2 b(t)^2 c(t)^2\ul(t)^2 & \leq &
    5\frac{c(t)}{c(\alpha)}\cdot\exp\left(\Phit(\lambda,\alpha,\beta)\right)\cdot\frac{b(t)^2}{b(\alpha)^2}
    \nonumber
    \\[0.5ex]
    & & \cdot\left[\ul'(\alpha)^2+\lambda^2 b(\alpha)^2 c(\alpha)^2\ul(\alpha)^2\right]
    \label{th:est-tar}
\end{eqnarray}
for every $t\in[\alpha,\beta]$, where
\begin{equation}
    \Phit(\lambda,\alpha,\beta):=
    \frac{3}{\lambda}
    \int_{\alpha}^{\beta}\frac{1}{b(\tau)}
    \left(\frac{|c''(\tau)|}{c(\tau)^2}+\frac{c'(\tau)^2}{c(\tau)^3}\right)d\tau.
    \nonumber
\end{equation}

\end{prop}

\begin{proof}


In the interval $[\alpha,\beta]$ we consider the Tarama-type energy
\begin{equation}
    \Et(t):=
    \frac{\ul'(t)^2}{c(t)}+
    \lambda^2 b(t)^2 c(t) \ul(t)^2+
    \frac{c'(t)}{c(t)^2}\ul(t)\ul'(t)+
    \frac{c'(t)^2}{c(t)^3}\ul(t)^2.
    \nonumber
\end{equation}

\paragraph{\textmd{\textit{Equivalence with simpler energies}}}

We claim that
\begin{equation}
    \frac{1}{2}\frac{\ul'(t)^2}{c(t)}+\lambda^2 b(t)^2 c(t) \ul(t)^2\leq
    \Et(t)\leq
    \frac{3}{2}\frac{\ul'(t)^2}{c(t)}+
    \frac{5}{2}\lambda^2 b(t)^2 c(t) \ul(t)^2
    \label{equiv:etar}
\end{equation}
for every $t\in[\alpha,\beta]$.

Indeed, from the classical inequality
\begin{equation}
    \left|\frac{c'(t)}{c(t)^2}\ul(t)\ul'(t)\right|\leq
    \frac{1}{2}\left(\frac{\ul'(t)^2}{c(t)}+\frac{c'(t)^2}{c(t)^3}\ul(t)^2\right)
    \nonumber
\end{equation}
we immediately deduce the estimate from below in (\ref{equiv:etar}), and the estimate from above
\begin{equation}
    \Et(t)\leq
    \frac{3}{2}\frac{\ul'(t)^2}{c(t)}+
    \lambda^2 b(t)^2 c(t) \ul(t)^2+
    \frac{3}{2}\frac{c'(t)^2}{c(t)^3}\ul(t)^2.
    \label{est:tar-1}
\end{equation}

On the other hand, from assumption (\ref{hp:tar-c'}) we know that
\begin{equation}
    \frac{c'(t)^2}{c(t)^3}\leq
    \lambda^2 b(t)^2 c(t).
    \nonumber
\end{equation}

Plugging this inequality into (\ref{est:tar-1}) we obtain the estimate from above in (\ref{equiv:etar}).

\paragraph{\textmd{\textit{Differential inequality}}}

We claim that
\begin{equation}
    \Et(t)\leq\frac{b(t)^2}{b(\alpha)^2}\cdot
    \exp\left(\Phit(\lambda,\alpha,\beta)\right)\cdot
    \Et(\alpha)
    \qquad
    \forall t\in[\alpha,\beta].
    \label{est:etar}
\end{equation}

In order to prove (\ref{est:etar}), we start by computing the time-derivative of $\Et$, which turns out to be
\begin{eqnarray}
    \Et'(t) & = &
    2\lambda^2 b(t)b'(t)c(t)\ul(t)^2+
    \frac{c''(t)}{c(t)^2}\ul(t)\ul'(t)+
    \frac{c'(t)}{c(t)^3}\left(2c''(t)-\frac{3c'(t)^2}{c(t)}\right)\ul(t)^2
    \nonumber
    \\
    & =: &
    L_1(t)+L_2(t)+L_3(t).
    \label{est:tar-L123}
\end{eqnarray}

Let us estimate the three terms. Since $b(t)>0$ and $b'(t)\geq 0$, from the estimate from below in (\ref{equiv:etar}) we deduce that
\begin{equation}
    L_1(t)=
    \frac{2b'(t)}{b(t)}\cdot\lambda^2 b(t)^2 c(t) \ul(t)^2\leq
    \frac{2b'(t)}{b(t)}\cdot\Et(t).
    \label{est:tar-L1}
\end{equation}

As for the second term, we observe that
\begin{equation}
    L_2(t)\leq\frac{1}{\lambda b(t)}\cdot\frac{|c''(t)|}{c(t)^2 }\cdot\frac{1}{2}
    \left(\frac{\ul'(t)^2}{c(t)}+\lambda^2 b(t)^2 c(t) \ul(t)^2\right),
    \nonumber
\end{equation}
and, hence, applying again the estimate from below in (\ref{equiv:etar}), we deduce that
\begin{equation}
    L_2\leq\frac{1}{\lambda b(t)}\cdot\frac{|c''(t)|}{c(t)^2}\cdot\Et(t).
    \label{est:tar-L2}
\end{equation}

Finally, applying once more the estimate from below in (\ref{equiv:etar}), and recalling (\ref{hp:tar-c'}), we obtain that
\begin{eqnarray}
    L_3(t) & \leq &
    \frac{|c'(t)|}{c(t)^4}\frac{1}{\lambda^2 b(t)^2}
    \left(2|c''(t)|+3\frac{c'(t)^2}{c(t)}\right)
    \lambda^2 b(t)^2 c(t) \ul(t)^2
    \nonumber
    \\
    & \leq &
    \frac{1}{\lambda b(t)}
    \left(2\frac{|c''(t)|}{c(t)^2}+3\frac{c'(t)^2}{c(t)^3}\right)\cdot
    \Et(t).
    \label{est:tar-L3}
\end{eqnarray}

Plugging~\eqref{est:tar-L1}, \eqref{est:tar-L2}, and~\eqref{est:tar-L3} into~\eqref{est:tar-L123}, we conclude that
\begin{equation}
    \Et'(t)\leq
    \frac{2b'(t)}{b(t)}\cdot\Et(t)+
    \frac{3}{\lambda b(t)}
    \left(\frac{|c''(t)|}{c(t)^2}+\frac{c'(t)^2}{c(t)^3}\right)\cdot
    \Et(t)
    \nonumber
\end{equation}
for every $t\in[\alpha,\beta]$. Integrating this differential inequality in $[\alpha,t]$, and recalling (\ref{eqn:int-b^2}), we obtain exactly (\ref{est:etar}).

\paragraph{\textmd{\textit{Conclusion}}}

From the estimate from below in~\eqref{equiv:etar} we deduce that
\begin{equation}
    \ul'(t)^2+\lambda^2 b(t)^2 c(t)^2\ul(t)^2
    \leq
    2c(t)\Et(t),
    \nonumber
\end{equation}
while from the estimate from above we deduce that
\begin{equation}
    \Et(\alpha)
    \leq
    \frac{5}{2c(\alpha)}
    \left(\ul'(\alpha)^2+
    \lambda^2 b(\alpha)^2 c(\alpha)^2\ul(\alpha)^2\right).
    \nonumber
\end{equation}
Combining the last two inequalities with~\eqref{est:etar}, we conclude that
\begin{eqnarray}
    \ul'(t)^2+\lambda^2 b(t)^2 c(t)^2\ul(t)^2
    & \leq &
    5\frac{c(t)}{c(\alpha)}\cdot
    \frac{b(t)^2}{b(\alpha)^2}\cdot
    \exp\left(\Phit(\lambda,\alpha,\beta)\right)
    \nonumber
    \\
    & &
    \cdot
    \left(\ul'(\alpha)^2+
    \lambda^2 b(\alpha)^2 c(\alpha)^2\ul(\alpha)^2\right),
    \nonumber
\end{eqnarray}
which is exactly~\eqref{th:est-tar}.
\end{proof}

The last result is a gluing statement.  Its purpose is to combine an estimate before the threshold, typically obtained by means of a Kowalevskian energy, with an estimate after the threshold, obtained for instance by means of a hyperbolic or Tarama-type energy. The conclusion is a global estimate in the standard energy.

\begin{prop}[Gluing estimates in two subintervals]\label{prop:mix}

Let $T$, $\nu_1$, and $\nu_2$ be three positive real numbers with $\nu_1\leq\nu_2$. Let $b:[0,T]\to[0,+\infty)$ be a nondecreasing continuous function, let $\lambda\geq 0$ be a real number, and let $\tl$ be defined by~\eqref{defn:tl-nu2}.

Let $v:[0,T]\to\re$ be a function of class $C^1$ with the following two properties.
\begin{enumerate}
\renewcommand{\labelenumi}{(\roman{enumi})}

    \item \emph{(Estimate before the threshold)} There exist two real numbers $H_1\geq 2$ and $d_0\in[\nu_1,\nu_2]$ such that
    \begin{equation}
        v'(t)^2+\lambda^2 b(\tl)^2 d_0^2\, v(t)^2\leq
        H_1\left[v'(0)^2+\lambda^2 b(\tl)^2 d_0^2\, v(0)^2\right]
        \label{hp:est-kow}
    \end{equation}
    for every $t\in[0,\tl]$.

    \item  \emph{(Estimate after the threshold)} There exists a real number $H_2\geq 1$ such that, if $\tl<T$, then there exists a function $d:[\tl,T]\to[\nu_1,\nu_2]$ with $d(\tl)=d_0$ and
    \begin{equation}
        v'(t)^2+\lambda^2 b(t)^2 d(t)^2 v(t)^2\leq
        H_2\frac{b(t)^2}{b(\tl)^2}\cdot
        \left[v'(\tl)^2+\lambda^2 b(\tl)^2 d(\tl)^2 v(\tl)^2\right]
        \label{hp:est-hyp}
    \end{equation}
    for every $t\in[\tl,T]$.

\end{enumerate}

Then the following global estimate holds for every $t\in[0,T]$:
\begin{equation}
    v'(t)^2+\lambda^2 v(t)^2\leq
    3\left(\frac{\nu_2}{\nu_1}\right)^2H_1 H_2\cdot
    \left(1+\nu_2^2 b(T)^2\right)\cdot
    \left[(1+\lambda^2 \tl^2)\cdot v'(0)^2+\lambda^2 v(0)^2\right].
    \label{th:est-kow-hyp}
\end{equation}

\end{prop}

\begin{proof}

Let us set, for simplicity,
\begin{equation}
    E_0:=v'(0)^2+\lambda^2 b(\tl)^2 d_0^2\, v(0)^2.
    \nonumber
\end{equation}
From the estimate before the threshold, namely~\eqref{hp:est-kow}, we deduce that
\begin{equation}
    v'(t)^2\leq
    H_1 E_0
    \qquad
    \forall t\in[0,\tl].
    \label{est:kow-v'}
\end{equation}

We now turn this estimate on $v'$ into an estimate on $v$. Since
\begin{equation}
|v(t)|\leq
|v(0)|+\int_{0}^{t}|v'(s)|\,ds,
    \nonumber
\end{equation}
by squaring this inequality and applying H\"older's inequality, we obtain that
\begin{equation}
v(t)^{2}\leq
2v(0)^{2}+2t\int_{0}^{t}|v'(s)|^{2}\,ds.
    \nonumber
\end{equation}
Thus, from~\eqref{est:kow-v'}, we deduce that
\begin{equation}
v(t)^{2}\leq 2v(0)^{2}+2 H_1 E_0 \cdot t^{2}
\qquad
\forall t\in[0,\tl].
\label{est:kow-v}
\end{equation}

Combining~\eqref{est:kow-v'} and~\eqref{est:kow-v}, we conclude that, for every $t\in[0,\tl]$,
\begin{eqnarray*}
v'(t)^{2}+\lambda^{2}v(t)^{2} & \leq &
H_1 E_0 +\lambda^{2}\left(2v(0)^{2}+2 H_1 E_0 t^{2}\right)
\\
& = &
H_1(1+2\lambda^{2}t^{2})\cdot v'(0)^{2}
\\
& &
+\left[2+H_1 b(\tl)^2 d_0^2+
2H_1 \lambda^{2} b(\tl)^2d_0^2\, t^2\right]
\cdot\lambda^{2}v(0)^{2}.
\end{eqnarray*}
Since $d_0\leq\nu_2$, from the definition of $\tl$ we know that
\begin{equation}
    \lambda^2 b(\tl)^2d_0^2\, t^2\leq
    \left(\nu_2\lambda\tl b(\tl)\strut\right)^2\leq
    1
    \qquad
    \forall t\in[0,\tl].
    \nonumber
\end{equation}
Recalling that $H_1\geq 2$, we obtain
\begin{equation}
    v'(t)^{2}+\lambda^{2}v(t)^{2}\leq
    H_1\left[(1+2\lambda^{2}t^{2})\cdot v'(0)^{2}+
    \left(3+\nu_2^2b(\tl)^{2}\right)\cdot\lambda^{2}v(0)^{2}\right].
    \nonumber
\end{equation}
Finally, since $t\leq\tl$ and $b(\tl)\leq b(T)$, we have
\[
    1+2\lambda^2t^2\leq 3(1+\lambda^2\tl^2)
    \qquad\text{and}\qquad
    3+\nu_2^2b(\tl)^2\leq 3(1+\nu_2^2b(T)^2).
\]
Taking also into account that $H_2\geq1$ and $\nu_2/\nu_1\geq1$, we conclude that \eqref{th:est-kow-hyp} is true for every $t\in[0,\tl]$.

If $\tl=T$, then there is nothing else to prove. Let us therefore assume that $\tl<T$. Since $b$ is continuous, the definition of $\tl$ gives
\begin{equation}
\nu_2\lambda\tl b(\tl)=1.
\label{eq:b-tl=1}
\end{equation}
In particular $b(\tl)>0$, and by monotonicity also $b(t)>0$ for every $t\in[\tl,T]$.

In the interval $[\tl,T]$ we use the estimate after the threshold, namely~\eqref{hp:est-hyp}, together with~\eqref{hp:est-kow} evaluated at $t=\tl$. Since $d(\tl)=d_0$, we obtain
\begin{eqnarray*}
\frac{b(t)^{2}d(t)^2}{1+b(t)^{2}d(t)^2}
\left(v'(t)^{2}+\lambda^{2}v(t)^{2}\right) & \leq &
v'(t)^{2}+\lambda^{2}b(t)^{2}d(t)^2 v(t)^{2}
\\
& \leq &
H_2\frac{b(t)^{2}}{b(\tl)^2}
\left(v'(\tl)^{2}+\lambda^{2}b(\tl)^2 d(\tl)^2 v(\tl)^{2}\right)
\\
& \leq &
H_1H_2\frac{b(t)^{2}}{b(\tl)^2}
\left(v'(0)^{2}+\lambda^{2}b(\tl)^2 d_0^2\, v(0)^{2}\right)
\\
& = &
H_1H_2\, b(t)^{2}
\left(\frac{v'(0)^{2}}{b(\tl)^{2}}+\lambda^{2}d_0^2\, v(0)^{2}\right),
\end{eqnarray*}
and as a consequence
\begin{equation}
    v'(t)^{2}+\lambda^{2}v(t)^{2}\leq
    H_1H_2\frac{1+b(t)^2d(t)^2}{d(t)^2}
    \left(\frac{v'(0)^{2}}{b(\tl)^{2}}+\lambda^{2}d_0^2\, v(0)^{2}\right).
    \nonumber
\end{equation}
Using~\eqref{eq:b-tl=1} and the inequalities
\[
    d_0\leq\nu_2,\qquad
    d(t)\geq\nu_1,\qquad
    b(t)\leq b(T),\qquad
    d(t)\leq\nu_2,
\]
we conclude that
\begin{equation}
    v'(t)^{2}+\lambda^{2}v(t)^{2}\leq
    H_1H_2
    \left(\frac{\nu_2}{\nu_1}\right)^2
    \left(1+\nu_2^2 b(T)^2\right)
    \left[\lambda^{2}\tl^{2}v'(0)^{2}+\lambda^{2}v(0)^{2}\right].
    \nonumber
\end{equation}
This is stronger than~\eqref{th:est-kow-hyp}, and therefore the proof is complete.
\end{proof}

\setcounter{equation}{0}
\section{Proofs of the main results and corollaries}\label{sec:proofs}

In this section we prove the main results stated in Section~\ref{sec:statements}. The proofs are obtained by combining the energy estimates of Section~\ref{sec:energy}, with suitable choices of the auxiliary coefficients and of the frequency-dependent splitting points.

The estimates of Section~\ref{sec:energy} have been proved under additional regularity assumptions on the monotone factor $b$. We now show that this entails no loss of generality. More precisely, we prove that any estimate obtained under the additional assumption that $b$ is smooth and nondecreasing extends to an arbitrary nondecreasing function $b$, by approximating $b$ from above with smooth nondecreasing functions.

Besides the convergence of the corresponding solutions, this approximation requires some care only in the frequency-dependent quantities: the thresholds and the minima defining the functions $\Phi_i$ have to pass to the limit. Once this reduction argument is established, in the proofs of the main results we shall freely assume that $b$ has all the additional regularity needed in order to apply the estimates of Section~\ref{sec:energy}.

\subsection{Approximation of the monotone factor}\label{sec:approx}

\paragraph{\textmd{\textit{Approximation of $b$}}}

We extend $b$ to the half-line $[0,+\infty)$ by setting $b(t)=b(T)$ for every $t\geq T$, and then for every $\ep>0$ we consider the convolution
\begin{equation}
    b_\ep(t):=\int_0^1 b(t+\ep s)\rho(s)\,ds
    \qquad
    \forall t\geq 0,
    \nonumber
\end{equation}
where $\rho$ is a smooth mollifier with support in $(0,1)$ and integral equal to~$1$. It is immediate that $b_\ep$ is smooth and nondecreasing, and satisfies
\begin{equation}
    b_\ep(t)\geq b(t)
    \qquad
    \forall t\in[0,T],
    \nonumber
\end{equation}
with equality at least when $t=T$ because of the extension beyond $T$. Moreover, $b_\ep$ converges to $b$ both in the sense that $b_\ep\to b$ in $L^1((0,T))$, and in the sense that
\begin{equation}
    \lim_{\ep\to 0^+}b_\ep(t)=b(t)
    \label{bep-pointwise}
\end{equation}
for every continuity point $t$ of $b$, namely for every $t\in[0,T]$ with the possible exception of a countable set of points.

\paragraph{\textmd{\textit{Convergence of solutions}}}

We now fix $\lambda$ and a solution $\ul$ to~(\ref{eqn:main-ode-c}). Let $u_{\lambda,\ep}$ be the solution to the same equation with $b_\ep$ in place of $b$, and with the same initial data as $\ul$. Since $b_\ep\to b$ in $L^1((0,T))$ and the functions $b_\ep$ are uniformly bounded, the coefficients $b_\ep(t)^2c(t)^2$ converge to $b(t)^2c(t)^2$ in $L^1((0,T))$. Therefore, the classical continuous-dependence result for linear ordinary differential equations gives
\[
    u_{\lambda,\ep}\to\ul
    \qquad
    \text{in }C^1([0,T])
    \quad\text{as }\ep\to0^+.
\]
This is enough to pass to the limit in the left-hand sides of~(\ref{th:b(t)-above}), (\ref{th:C0-theory}), (\ref{th:C1-theory}), and~(\ref{th:C2-theory}).

\paragraph{\textmd{\textit{Approximation of thresholds and minimum values}}}

Let us now consider the corresponding right-hand sides, which involve quantities such as $b_\ep(T)$, $t_{\lambda,\ep}$, $A_{\lambda,\ep}$, $\Phi_{0,\ep}(\lambda,T)$, $\Phi_{1,\ep}(\lambda,T)$, $\Phi_{2,\ep}(\lambda,T)$, defined as in Section~\ref{sec:statements}, just with $b_\ep$ instead of $b$. We claim that, as $\ep\to 0^+$,
\begin{equation}
    b_\ep(T)\to b(T),
    \qquad\quad
    t_{\lambda,\ep}\to\tl,
    \qquad\quad
    A_{\lambda,\ep}\to A_\lambda,
    \label{eqn:lim-ep-bt}
\end{equation}
and
\begin{equation}
    \Phi_{i,\ep}(\lambda,T)\to\Phi_{i}(\lambda,T)
    \qquad
    \forall i\in\{0,1,2\}.
    \label{eqn:lim-ep-Phi}
\end{equation}

Once these convergences are established, the estimates satisfied by the approximate solutions $u_{\lambda,\ep}$ pass to the limit and yield the desired estimates for the solution $\ul$ of the original equation.

The first convergence in (\ref{eqn:lim-ep-bt}) is trivial because actually $b_\ep(T)=b(T)$ for every $\ep>0$.

\paragraph{\textmd{\textit{Convergence of $t_{\lambda,\ep}$}}}

For simplicity, we present the argument in the case where $\tl$ is defined by~(\ref{defn:tl}); the case where it is defined by~(\ref{defn:tl-nu2}) is completely analogous. Since $b_\ep\geq b$ in $[0,T]$, the definition of the threshold gives $t_{\lambda,\ep}\leq\tl$ for every $\ep>0$, and hence,
\begin{equation}
    \limsup_{\ep\to0^+}t_{\lambda,\ep}\leq\tl.
    \label{est:limsup-tleps}
\end{equation}
Conversely, let $t<\tl$ be a continuity point of $b$. Then $\lambda t b(t)<1$, and by the pointwise convergence~(\ref{bep-pointwise}) we have $\lambda t b_\ep(t)<1$ for every $\ep$ small enough. Hence, $t<t_{\lambda,\ep}$ for every $\ep$ small enough, and therefore,
\begin{equation}
    \liminf_{\ep\to0^+}t_{\lambda,\ep}\geq t.
    \nonumber
\end{equation}
Letting $t\to\tl^-$ along continuity points of $b$, which are dense in $[0,T]$, we obtain
\begin{equation}
    \liminf_{\ep\to0^+}t_{\lambda,\ep}\geq\tl.
    \label{est:liminf-tleps}
\end{equation}
Combining~(\ref{est:limsup-tleps}) and~(\ref{est:liminf-tleps}), we conclude that $t_{\lambda,\ep}\to\tl$.

\paragraph{\textmd{\textit{Convergence of $A_{\lambda,\ep}$}}}

Since $b_\ep\geq b$ in $[0,T]$, and $t_{\lambda,\ep}\leq\tl$, it follows that the approximate bad set is contained in the original bad set, and hence
\begin{equation}
    \limsup_{\ep\to0^+}A_{\lambda,\ep}\leq A_\lambda.
    \label{est:limsup-Aleps}
\end{equation}

Let us prove the opposite inequality. If the bad set
\[
    \left\{t\in[\tl,T]:
    |c'(t)|>\lambda b(t)c(t)^2
    \right\}
\]
is empty, then $A_\lambda=\tl$, and the conclusion follows from
$t_{\lambda,\ep}\leq A_{\lambda,\ep}$ and $t_{\lambda,\ep}\to\tl$.

Otherwise, let $t$ be a point of the bad set. Since the inequality
\[
    |c'(t)|>\lambda b(t)c(t)^2
\]
is strict, and since $c$ and $c'$ are continuous while $b$ is nondecreasing,
the same inequality holds in a left neighbourhood of $t$. Indeed, moving
slightly to the left, the quantities involving $c$ change continuously, while
$b$ can only decrease. Hence, this left neighbourhood contains continuity
points $\sigma$ of $b$ which still belong to the bad set.

For every such point $\sigma$, the pointwise convergence
$b_\ep(\sigma)\to b(\sigma)$ gives
\[
    |c'(\sigma)|>\lambda b_\ep(\sigma)c(\sigma)^2
\]
for every $\ep$ small enough. Thus $\sigma\leq A_{\lambda,\ep}$ for every
$\ep$ small enough. Letting first $\sigma$ approach $t$ from the left, and
then letting $t\to A_\lambda^-$ along points of the bad set, we obtain
\begin{equation}
    \liminf_{\ep\to0^+}A_{\lambda,\ep}\geq A_\lambda.
    \nonumber
\end{equation}
Together with~(\ref{est:limsup-Aleps}), this proves that $A_{\lambda,\ep}\to A_\lambda$.

\paragraph{\textmd{\textit{Convergence of minimum values}}}

Finally, let us consider the limits in~(\ref{eqn:lim-ep-Phi}). The case $i=0$ is immediate from~(\ref{defn:Phi0}), because $b_\ep\to b$ in $L^1((0,T))$.

For the cases $i=1$ and $i=2$, we use the following fact, which follows from a standard compactness argument. Assume that $\{f_n\}$ is a sequence of functions, and $\{K_n\}$ is a sequence of compact sets in some metric space. Assume that
\begin{itemize}
\item for every $n$ the function $f_n$ attains its minimum in $K_n$,

\item the sequence $\{K_n\}$ converges to a compact set $K_\infty$ in the Hausdorff sense,

\item the sequence $\{f_n\}$ converges to a function $f_\infty:K_\infty\to\re$ in the sense that $f_n(x_n)\to f_\infty(x_\infty)$ for every sequence $x_n\to x_\infty$ with $x_n\in K_n$ for each $n$.
\end{itemize}
Then one has convergence of the minimum values, namely
\[
    \min\{f_n(x):x\in K_n\}\to
    \min\{f_\infty(x):x\in K_\infty\}.
\]

We apply this result to the functions $\psi_{1,\lambda,\ep}(s)$ and $\psi_{2,\lambda,\ep}(s,a)$ defined in analogy with~(\ref{defn:psi-1}) and~(\ref{defn:psi-2}). The sets where the minimum values are computed are intervals in the case $i=1$, and polygons in the case $i=2$, and in both cases the vertices are converging to the vertices of the limit polygon. As for the functions, in the case of $\psi_{1,\lambda,\ep}$ and of the first two terms of $\psi_{2,\lambda,\ep}$ the required convergence along convergent sequences of points in the moving domains follows from the $L^1$ convergence of $b_\ep$ to $b$, from the convergence $t_{\lambda,\ep}\to\tl$, and from the continuity of the oscillation with respect to $s$.

In the case of the third term in the definition of $\psi_{2,\lambda,\ep}$, from the definition of $t_{\lambda,\ep}$ and the monotonicity of $t\mapsto tb_\ep(t)$ we deduce that
\[
    \frac{1}{b_\ep(t)}\leq \nu_2\lambda t\leq\nu_2\lambda T
    \qquad
    \forall t\in(t_{\lambda,\ep},T).
\]
This provides a uniform bound from above for the factor $1/b_\ep(t)$ in the third term of $\psi_{2,\lambda,\ep}$. Since the remaining factor involving $c'$ and $c''$ is continuous, hence bounded, on the relevant compact subintervals of $(0,T]$, the convergence of the third term follows by dominated convergence, taking also into account the convergence of the endpoints.

This proves the convergence of minimum values also for $i=1$ and $i=2$.
\qed


\subsection{Proof of Theorem~\ref{thm:b-monotone}}

In this proof there are no oscillations to be controlled, and therefore, the only effect to be estimated is the one produced by the monotone factor $b(t)$. We first obtain the estimate from above by combining the Kowalevskian estimate before the threshold with the hyperbolic estimate after the threshold. Then we show that the factor $\lambda^2\tl^2$ in the loss of regularity cannot be improved in general.

\paragraph{\textmd{\textit{Estimate from above}}}

To begin with, we apply Proposition~\ref{prop:kow} in the interval $[0,\tl]$ with
\[
    \nu_1=d_0=\nu_2:=1,
    \qquad
    c(t)\equiv 1.
\]
We obtain
\begin{equation}
    \ul'(t)^2+\lambda^2 b(\tl)^2 \ul(t)^2\leq
    \exp(1)\cdot\left[\ul'(0)^2+\lambda^2 b(\tl)^2 \ul(0)^2\right]
    \qquad
    \forall t\in[0,\tl].
    \label{est:b-kow}
\end{equation}

If $\tl<T$, then we apply Proposition~\ref{prop:hyp} with
\begin{equation}
    [\alpha,\beta]:=[\tl,T],
    \qquad\quad
    \nu_1=\nu_2:=1,
    \qquad\quad
    d(t)\equiv 1.
    \nonumber
\end{equation}
In this case the quantity $\Phih(\lambda,\tl,T)$ defined in~\eqref{defn:Phih} is equal to~$0$, and therefore
\begin{equation}
    \ul'(t)^2+\lambda^2 b(t)^2 \ul(t)^2\leq
    \frac{b(t)^2}{b(\tl)^2}\cdot
    \left[\ul'(\tl)^2+\lambda^2 b(\tl)^2 \ul(\tl)^2\right]
    \qquad
    \forall t\in[\tl,T].
    \label{est:b-hyp}
\end{equation}

Due to~\eqref{est:b-kow} and~\eqref{est:b-hyp}, the function $v(t):=\ul(t)$ satisfies the assumptions of Proposition~\ref{prop:mix} with
\[
    H_1:=\exp(1),
    \qquad
    H_2:=1,
    \qquad
    \nu_1=d_0=\nu_2:=1,
    \qquad
    d(t)\equiv 1.
\]
With these choices, estimate~\eqref{th:est-kow-hyp} gives
\begin{equation}
    \ul'(t)^2+\lambda^2 \ul(t)^2\leq
    3\exp(1)\cdot\left(1+b(T)^2\right)\cdot
    \left[(1+\lambda^2 \tl^2)\cdot\ul'(0)^2+\lambda^2\ul(0)^2\right]
    \nonumber
\end{equation}
for every $t\in[0,T]$, which is exactly~\eqref{th:b(t)-above}.

\paragraph{\textmd{\textit{Optimality}}}

Let $\ul=\ul(t)$ denote the solution to~\eqref{eqn:main-ode} with initial data
\begin{equation}
\ul(0)=0,
\qquad\qquad
\ul'(0)=1.
\label{eqn:optimality-data}
\end{equation}
We claim that
\begin{equation}
0\leq\ul(t)\leq\tl
    \qquad\text{and}\qquad
    \frac{1}{2}\leq\ul'(t)\leq 1
    \qquad
    \forall t\in[0,\tl/2].
    \label{claim:optimality}
\end{equation}

To this end, we set
\begin{equation}
    S:=\sup\left\{s\in[0,\tl/2]:
    \forall t\in[0,s]\quad
    0\leq\ul(t)\leq\tl
    \text{ and }
    1/2\leq\ul'(t)\leq 3/2\strut
    \right\}.
    \nonumber
\end{equation}
Thus proving~\eqref{claim:optimality} is equivalent to showing that $S=\tl/2$.

The initial conditions~\eqref{eqn:optimality-data} guarantee that the inequalities in~\eqref{claim:optimality} are true in a right neighborhood of $0$, and hence, $S>0$. Moreover, by the definition of $S$ we have
\begin{equation}
    0\leq\ul(s)\leq\tl
    \qquad\text{and}\qquad
    \frac{1}{2}\leq\ul'(s)\leq\frac{3}{2}
    \qquad
    \forall s\in[0,S].
    \label{est:S}
\end{equation}

Assume by contradiction that $S<\tl/2$. By the maximality of $S$, one of the following alternatives must occur:
\begin{equation}
    \ul(S)=\tl
    \qquad\text{or}\qquad
    \ul'(S)=\frac{1}{2}.
    \label{opt-alternative}
\end{equation}
Indeed, the other two possibilities, namely $\ul(S)=0$ and $\ul'(S)=3/2$, are ruled out by the fact that $\ul$ is increasing and $\ul'$ is nonincreasing in $[0,S]$.

Since $\ul'$ is nonincreasing in $[0,S]$, we deduce that
\begin{equation}
    \ul(S)=\int_{0}^{S}\ul'(s)\,ds\leq
    S<\tl,
    \nonumber
\end{equation}
which rules out the first possibility in~\eqref{opt-alternative}.

Similarly, from~\eqref{est:S} and from the definition of $\tl$ we obtain, for every $s\in[0,S]$,
\begin{equation}
    \ul''(s)=
    -\lambda^{2}b(s)^{2}\ul(s)\geq
    -\lambda^{2}b(\tl)^{2}\tl\geq
    -\frac{1}{\tl}.
    \nonumber
\end{equation}
Therefore
\begin{equation}
    \ul'(S)=
    1+\int_{0}^{S}\ul''(s)\,ds\geq
    1-\frac{S}{\tl}>
    \frac{1}{2},
    \nonumber
\end{equation}
which rules out also the second possibility in~\eqref{opt-alternative}. This proves~\eqref{claim:optimality}.

Once we know that $\ul'(t)\geq 1/2$ for every $t\in[0,\tl/2]$, we deduce
\begin{equation}
\ul(\tl/2)=
\int_{0}^{\tl/2}\ul'(s)\,ds\geq
\tl/4.
    \nonumber
\end{equation}
At this point, if~\eqref{hp:optimality} holds for every solution to~\eqref{eqn:main-ode} and every $t\in[0,T]$, then it holds in particular for the solution $\ul$ with initial data~\eqref{eqn:optimality-data} and for $t=\tl/2$. Hence,
\begin{equation}
    \frac{\lambda^{2}\tl^{2}}{16}\leq
    \ul'(\tl/2)^{2}+\lambda^{2}\ul(\tl/2)^{2}\leq
    c_\lambda,
    \nonumber
\end{equation}
which proves~\eqref{th:b(t)-optimal}.
\qed

\subsection{Proof of Theorem~\ref{thm:c0}}

The proof is based on the same gluing argument as in the purely degenerate case. The only additional ingredient is a regularization of the oscillating factor $c$ at scale $1/\lambda$, which provides the auxiliary coefficient used in the hyperbolic energy.

To this end, we extend $c$ to the whole half-line $[0,+\infty)$ by setting
\begin{equation}
    c(t):=c(T)
    \qquad
    \forall t\geq T.
    \nonumber
\end{equation}
For every $\lambda>0$ we introduce the local average
\begin{equation}
    \cl(t):=\lambda\int_{t}^{t+1/\lambda}c(s)\,ds
    \qquad
    \forall t\geq 0.
    \nonumber
\end{equation}

From the strict hyperbolicity assumption~(\ref{hp:c-sh}) it follows that
\begin{equation}
    \nu_1\leq\cl(t)\leq\nu_2
    \qquad
    \forall t\geq 0.
    \label{est:gal-sh}
\end{equation}
Moreover, from the oscillation control~(\ref{defn:c-omega}) we obtain
\begin{equation}
    |\cl'(t)|\leq\lambda\,\omega(1/\lambda)
    \qquad
    \forall t\geq 0,
    \label{est:gal'}
\end{equation}
and
\begin{equation}
    \left|\cl(t)-c(t)\right|\leq\omega(1/\lambda)
    \qquad
    \forall t\geq 0.
    \label{est:gal-c}
\end{equation}

Let now $\ul$ be any solution to~\eqref{eqn:main-ode-c}. In the interval $[0,\tl]$ we apply Proposition~\ref{prop:kow} with $d_0:=\cl(\tl)$. We obtain
\begin{equation}
    \ul'(t)^2+\lambda^2 b(\tl)^2 \cl(\tl)^2 \ul(t)^2\leq
    \exp\left(\frac{\nu_2}{\nu_1}\right)\cdot
    \left[\ul'(0)^2+\lambda^2 b(\tl)^2 \cl(\tl)^2 \ul(0)^2\right]
    \label{est:b-c0-kow}
\end{equation}
for every $t\in[0,\tl]$.

If $\tl<T$, then we apply Proposition~\ref{prop:hyp} with $[\alpha,\beta]:=[\tl,T]$ and $d(t):=\cl(t)$. This choice is admissible because of~\eqref{est:gal-sh}, and it is compatible with Proposition~\ref{prop:mix} because $d(\tl)=\cl(\tl)=d_0$. Thanks to~\eqref{est:gal'} and~\eqref{est:gal-c}, we have
\begin{equation}
    \int_{\tl}^T |\cl'(\tau)|\,d\tau\leq
    \lambda\omega\left(\frac{1}{\lambda}\right)T,
    \nonumber
\end{equation}
and
\begin{equation}
    \sup_{t\in[\tl,T]}|\cl(t)-c(t)|\leq
    \omega\left(\frac{1}{\lambda}\right).
    \nonumber
\end{equation}
Therefore, the quantities $\Phih(\lambda,\tl,T)$ and $\Phi_0(\lambda,T)$ defined in~\eqref{defn:Phih} and~\eqref{defn:Phi0}, respectively, satisfy
\begin{equation}
    \Phih(\lambda,\tl,T)\leq\Phi_0(\lambda,T).
    \nonumber
\end{equation}
Hence, from Proposition~\ref{prop:hyp} we deduce that
\begin{eqnarray}
    \ul'(t)^2+\lambda^2 b(t)^2 \cl(t)^2\ul(t)^2 & \leq &
    \exp(\Phi_0(\lambda,T))\cdot\frac{b(t)^2}{b(\tl)^2}
    \nonumber
    \\
    & &
    \cdot\left[\ul'(\tl)^2+\lambda^2 b(\tl)^2 \cl(\tl)^2 \ul(\tl)^2\right]
    \label{est:b-c0-hyp}
\end{eqnarray}
for every $t\in[\tl,T]$.

Due to~\eqref{est:b-c0-kow} and~\eqref{est:b-c0-hyp}, the function $v(t):=\ul(t)$ satisfies the assumptions of Proposition~\ref{prop:mix} with
\[
    H_1:=\exp\left(\frac{\nu_2}{\nu_1}\right),
    \qquad
    H_2:=\exp(\Phi_0(\lambda,T)),
    \qquad
    d_0:=\cl(\tl),
    \qquad
    d(t):=\cl(t).
\]
With these choices, estimate~\eqref{th:est-kow-hyp} gives
\begin{eqnarray*}
    \ul'(t)^2+\lambda^2 \ul(t)^2 & \leq &
    3\left(\frac{\nu_2}{\nu_1}\right)^2
    \exp\left(\frac{\nu_2}{\nu_1}\right)
    \exp(\Phi_0(\lambda,T))
    \left(1+\nu_2^2 b(T)^2\right)
    \\
    & &
    \cdot
    \left[(1+\lambda^2 \tl^2)\cdot\ul'(0)^2+
    \lambda^2\ul(0)^2\right]
\end{eqnarray*}
for every $t\in[0,T]$, which is exactly~\eqref{th:C0-theory}.
\qed

\begin{rmk}[When $c$ is not continuous]
\begin{em}

In the proof we used the fact that the averaged coefficient $\cl$ is of class $C^1$, which is immediate when $c$ is continuous. In the general case one can justify the argument by a standard approximation. After extending $c$ outside
$[0,T]$ without increasing its oscillation function, we set $c_\ep:=\rho_\ep*c$, where $\rho_\ep$ is a nonnegative convolution kernel with integral equal to~$1$.

Then $c_\ep$ still satisfies the strict hyperbolicity bounds, and the oscillation of
$c_\ep$ is controlled by the same function $\omega$. Therefore the estimate proved above applies to the approximating coefficients $c_\ep$ with the same right-hand side. Passing to the limit as $\ep\to0^+$, by the continuous dependence of solutions on the coefficients, gives the result for the original coefficient $c$.

\end{em}
\end{rmk}


\subsection{Proof of Theorem~\ref{thm:c1}}

The proof follows the same general strategy as in the previous cases. We use the Kowalevskian estimate before the threshold $\tl$, and then a hyperbolic estimate after the threshold. The new feature is that, after $\tl$, the hyperbolic estimate is itself obtained by splitting the interval in a point $\ssl$ which realizes the minimum in the definition of $\Phi_1(\lambda,T)$.


Let $\tl$ be defined by~\eqref{defn:tl-nu2}, and let $\ssl$ be a minimizer in \eqref{defn:Phi1} in the definition of $\Phi_1(\lambda,T)$.

In the interval $[0,\tl]$ we apply Proposition~\ref{prop:kow} with $d_0:=c(\ssl)$. We obtain
\begin{equation}
    \ul'(t)^2+\lambda^2 b(\tl)^2 c(\ssl)^2 \ul(t)^2\leq
    \exp\left(\frac{\nu_2}{\nu_1}\right)\cdot
    \left[\ul'(0)^2+\lambda^2 b(\tl)^2 c(\ssl)^2 \ul(0)^2\right]
    \label{est:c1-kow}
\end{equation}
for every $t\in[0,\tl]$.

If $\tl<T$, we claim that
\begin{eqnarray}
    \ul'(t)^2+\lambda^2 b(t)^2 d(t)^2\ul(t)^2
    & \leq &
    \exp(\Phi_1(\lambda,T))\cdot\frac{b(t)^2}{b(\tl)^2}
    \nonumber
    \\
    & &
    \cdot
    \left[\ul'(\tl)^2+\lambda^2 b(\tl)^2 c(\ssl)^2 \ul(\tl)^2\right]
    \label{est:c1-hyp}
\end{eqnarray}
for every $t\in[\tl,T]$, where
\begin{equation}
    d(t):=
    \begin{cases}
    c(\ssl) & \text{if }t\in[\tl,\ssl],
    \\
    c(t) & \text{if }t\in[\ssl,T].
    \end{cases}
    \label{defn:d-C1}
\end{equation}
This choice is compatible with Proposition~\ref{prop:mix}, because $d(\tl)=c(\ssl)=d_0$.

Let us prove~\eqref{est:c1-hyp}. We argue separately in the two intervals $[\tl,\ssl]$ and $[\ssl,T]$.

If $\tl<\ssl$, then in the interval $[\tl,\ssl]$ we apply Proposition~\ref{prop:hyp} with
\[
    [\alpha,\beta]:=[\tl,\ssl],
    \qquad
    d(t)\equiv c(\ssl).
\]
In this interval $d'(t)\equiv0$, and for every $s\in[\tl,\ssl]$ we have
\begin{equation}
    |d(s)-c(s)|=
    |c(\ssl)-c(s)|
    \leq
    \osc(c,(0,\ssl]).
    \nonumber
\end{equation}
Therefore,
\begin{equation}
    \Phih(\lambda,\tl,\ssl)
    \leq
    \frac{2\nu_2}{\nu_1}\lambda\cdot\osc(c,(0,\ssl])
    \int_{\tl}^{\ssl}b(t)\,dt
    =:\Phi_{1,1}(\lambda).
    \nonumber
\end{equation}
From Proposition~\ref{prop:hyp} we deduce
\begin{eqnarray}
    \ul'(t)^2+\lambda^2 b(t)^2 c(\ssl)^2\ul(t)^2
    & \leq &
    \exp(\Phi_{1,1}(\lambda))\cdot\frac{b(t)^2}{b(\tl)^2}
    \nonumber
    \\
    & &
    \cdot
    \left[\ul'(\tl)^2+\lambda^2 b(\tl)^2 c(\ssl)^2 \ul(\tl)^2\right]
    \label{est:c1-hyp-ts}
\end{eqnarray}
for every $t\in[\tl,\ssl]$.

If $\ssl<T$, then in the interval $[\ssl,T]$ we apply Proposition~\ref{prop:hyp} with
\[
    [\alpha,\beta]:=[\ssl,T],
    \qquad
    d(t):=c(t).
\]
In this case $d(t)=c(t)$, and therefore, the second term in the definition of $\Phih$ vanishes. Moreover, by the derivative bound on $c$, we have
\begin{equation}
    \Phih(\lambda,\ssl,T)
    =
    \frac{2}{\nu_1}\int_{\ssl}^{T}|c'(t)|\,dt
    =:\Phi_{1,2}(\lambda).
    \nonumber
\end{equation}
From Proposition~\ref{prop:hyp} we obtain
\begin{eqnarray}
    \ul'(t)^2+\lambda^2 b(t)^2 c(t)^2\ul(t)^2
    & \leq &
    \exp(\Phi_{1,2}(\lambda))\cdot\frac{b(t)^2}{b(\ssl)^2}
    \nonumber
    \\
    & &
    \cdot
    \left[\ul'(\ssl)^2+\lambda^2 b(\ssl)^2 c(\ssl)^2 \ul(\ssl)^2\right]
    \label{est:c1-hyp-sT}
\end{eqnarray}
for every $t\in[\ssl,T]$.

If $\tl<\ssl<T$, the last factor in the right-hand side of~(\ref{est:c1-hyp-sT}) can be estimated by applying~(\ref{est:c1-hyp-ts}) with $t=\ssl$. If $\ssl=\tl$, the same conclusion follows directly from~(\ref{est:c1-hyp-sT}), because in this case the initial point of the hyperbolic estimate is already $\tl$. We obtain
\begin{eqnarray}
    \ul'(t)^2+\lambda^2 b(t)^2 c(t)^2\ul(t)^2
    & \leq &
    \exp(\Phi_{1,1}(\lambda))\cdot
    \exp(\Phi_{1,2}(\lambda))\cdot
    \frac{b(t)^2}{b(\tl)^2}
    \nonumber
    \\
    & &
    \cdot
    \left[\ul'(\tl)^2+\lambda^2 b(\tl)^2 c(\ssl)^2 \ul(\tl)^2\right]
    \label{est:c1-hyp-combined}
\end{eqnarray}
for every $t\in[\ssl,T]$.

By the definition of $\ssl$ we have
\begin{equation}
    \Phi_{1,1}(\lambda)+\Phi_{1,2}(\lambda)
    =
    \psi_{1,\lambda}(\ssl)=
    \Phi_1(\lambda,T).
    \nonumber
\end{equation}
Combining this identity with~\eqref{est:c1-hyp-ts} and~\eqref{est:c1-hyp-combined}, we obtain~\eqref{est:c1-hyp} for every $t\in[\tl,T]$.

Due to~(\ref{est:c1-kow}) and~(\ref{est:c1-hyp}), the function $v(t):=\ul(t)$ satisfies the assumptions of Proposition~\ref{prop:mix} with
\[
    H_1:=\exp\left(\frac{\nu_2}{\nu_1}\right),
    \qquad
    H_2:=\exp(\Phi_1(\lambda,T)),
    \qquad
    d_0:=c(\ssl),
\]
and $d(t)$ defined by (\ref{defn:d-C1}). With these choices, estimate~\eqref{th:est-kow-hyp} gives
\begin{eqnarray*}
    \ul'(t)^2+\lambda^2 \ul(t)^2
    & \leq &
    3\left(\frac{\nu_2}{\nu_1}\right)^2
    \exp\left(\frac{\nu_2}{\nu_1}\right)
    \exp(\Phi_1(\lambda,T))
    \left(1+\nu_2^2 b(T)^2\right)
    \\
    & &
    \cdot
    \left[(1+\lambda^2 \tl^2)\cdot\ul'(0)^2+
    \lambda^2\ul(0)^2\right]
\end{eqnarray*}
for every $t\in[0,T]$, which is exactly~\eqref{th:C1-theory}.
\qed

\subsection{Proof of Corollary~\ref{cor:natural-scale-c1}}

Since $\lambda\geq 1$, there exists $\tau_\lambda\in(0,T]$ such that
\begin{equation}
    B(\tau_\lambda)=\frac{B(T)}{\lambda}.
    \label{defn:taul}
\end{equation}

Now we distinguish two cases. If $\tau_\lambda\geq\tl$, then $s=\tau_\lambda$ is an admissible competitor in (\ref{defn:Phi1}), and therefore,
\begin{equation}
    \Phi_1(\lambda,T)\leq
    \psi_{1,\lambda}(\tau_\lambda)\leq
    \frac{2\nu_2}{\nu_1}(\nu_2-\nu_1)\lambda B(\tau_\lambda)+
    \frac{2M}{\nu_1}\log\frac{B(T)}{B(\tau_\lambda)},
    \nonumber
\end{equation}
which implies (\ref{th:cor-c1}) because of (\ref{defn:taul}).

If $\tau_\lambda<\tl$, then we use $s=\tl$ as a competitor, and we find that
\begin{equation}
    \Phi_1(\lambda,T)\leq
    \psi_{1,\lambda}(\tl)=
    \frac{2M}{\nu_1}\log\frac{B(T)}{B(\tl)}\leq
    \frac{2M}{\nu_1}\log\frac{B(T)}{B(\tau_\lambda)}=
    \frac{2M}{\nu_1}\log\lambda,
    \nonumber
\end{equation}
which implies (\ref{th:cor-c1}) also in this second case.
\qed


\subsection{Proof of Theorem~\ref{thm:c2}}

The proof follows the same scheme as the proof of Theorem~\ref{thm:c1}, but with one additional splitting point. Again we use the Kowalevskian estimate before the threshold $\tl$, and after $\tl$ we obtain the required estimate by splitting the interval in two points $\ssl$ and $\al$, where $(\ssl,\al)$ is a minimizer in \eqref{defn:Phi2} in the definition of $\Phi_2(\lambda,T)$.

In the interval $[0,\tl]$ we apply Proposition~\ref{prop:kow} with $d_0:=c(\ssl)$. We obtain
\begin{equation}
    \ul'(t)^2+\lambda^2 b(\tl)^2 c(\ssl)^2 \ul(t)^2\leq
    \exp\left(\frac{\nu_2}{\nu_1}\right)
    \left[\ul'(0)^2+\lambda^2 b(\tl)^2 c(\ssl)^2 \ul(0)^2\right]
    \label{est:c2-kow}
\end{equation}
for every $t\in[0,\tl]$.

If $\tl<T$, we define $d(t)$ as in (\ref{defn:d-C1}), and we observe that this choice is compatible with Proposition~\ref{prop:mix}, because $d(\tl)=c(\ssl)=d_0$. We claim that
\begin{eqnarray}
    \ul'(t)^2+\lambda^2 b(t)^2d(t)^2\ul(t)^2
    & \leq &
    \frac{5\nu_2}{\nu_1}
    \exp(\Phi_2(\lambda,T))\cdot
    \frac{b(t)^2}{b(\tl)^2}
    \nonumber
    \\
    & &
    \cdot
    \left[\ul'(\tl)^2+
    \lambda^2 b(\tl)^2 c(\ssl)^2 \ul(\tl)^2\right]
    \label{est:c2-after}
\end{eqnarray}
for every $t\in[\tl,T]$.

In order to prove~\eqref{est:c2-after}, we argue separately in the three intervals
\[
    [\tl,\ssl],
    \qquad\quad
    [\ssl,\al],
    \qquad\quad
    [\al,T].
\]

If $\tl<\ssl$, then in the interval $[\tl,\ssl]$ we apply Proposition~\ref{prop:hyp} with
\[
    [\alpha,\beta]:=[\tl,\ssl],
    \qquad
    d(t)\equiv c(\ssl),
\]
and as in the proof of Theorem~\ref{thm:c1} we obtain that
\begin{equation}
    \Phih(\lambda,\tl,\ssl)
    \leq
    \frac{2\nu_2}{\nu_1}
    \lambda\cdot\osc(c,(0,\ssl])
    \int_{\tl}^{\ssl}b(t)\,dt
    =:\Phi_{2,1}(\lambda),
    \nonumber
\end{equation}
and therefore,
\begin{eqnarray}
    \ul'(t)^2+\lambda^2 b(t)^2 c(\ssl)^2\ul(t)^2
    & \leq &
    \exp(\Phi_{2,1}(\lambda))\cdot
    \frac{b(t)^2}{b(\tl)^2}
    \nonumber
    \\
    & &
    \cdot
    \left[\ul'(\tl)^2+
    \lambda^2 b(\tl)^2 c(\ssl)^2 \ul(\tl)^2\right]
    \label{est:c2-ts}
\end{eqnarray}
for every $t\in[\tl,\ssl]$. Since $\Phi_{2,1}(\lambda)\leq\psi_{2,\lambda}(\ssl,\al)=\Phi_2(\lambda,T)$, this establishes (\ref{est:c2-after}) in the interval $[\tl,\ssl]$.

If $\ssl<\al$, then in the interval $[\ssl,\al]$ we apply Proposition~\ref{prop:hyp} with
\[
    [\alpha,\beta]:=[\ssl,\al],
    \qquad
    d(t):=c(t),
\]
and as in the proof of Theorem~\ref{thm:c1} we obtain that
\begin{equation}
    \Phih(\lambda,\ssl,\al)
    =
    \frac{2}{\nu_1}\int_{\ssl}^{\al}|c'(t)|\,dt
    =:\Phi_{2,2}(\lambda),
    \nonumber
\end{equation}
and therefore,
\begin{eqnarray}
    \ul'(t)^2+\lambda^2 b(t)^2 c(t)^2\ul(t)^2
    & \leq &
    \exp(\Phi_{2,2}(\lambda))\cdot
    \frac{b(t)^2}{b(\ssl)^2}
    \nonumber
    \\
    & &
    \cdot
    \left[\ul'(\ssl)^2+
    \lambda^2 b(\ssl)^2 c(\ssl)^2 \ul(\ssl)^2\right]
    \label{est:c2-sa-local}
\end{eqnarray}
for every $t\in[\ssl,\al]$.
Estimating the last factor in the right-hand side of~\eqref{est:c2-sa-local} by means of~\eqref{est:c2-ts} with $t=\ssl$, we obtain
\begin{eqnarray}
    \ul'(t)^2+\lambda^2 b(t)^2 c(t)^2\ul(t)^2
    & \leq &
    \exp(\Phi_{2,1}(\lambda))\cdot
    \exp(\Phi_{2,2}(\lambda))\cdot
    \frac{b(t)^2}{b(\tl)^2}
    \nonumber
    \\
    & &
    \cdot
    \left[\ul'(\tl)^2+
    \lambda^2 b(\tl)^2 c(\ssl)^2 \ul(\tl)^2\right]
    \label{est:c2-sa}
\end{eqnarray}
for every $t\in[\ssl,\al]$. Since $\Phi_{2,1}(\lambda)+\Phi_{2,2}(\lambda)\leq\psi_{2,\lambda}(\ssl,\al)=\Phi_2(\lambda,T)$, this establishes (\ref{est:c2-after}) in the interval $[\ssl,\al]$.

Finally, if $\al<T$, then in the interval $[\al,T]$ we apply Proposition~\ref{prop:tarama} with
\[
    [\alpha,\beta]:=[\al,T].
\]
This is admissible because $\al\geq A_\lambda$, with $A_\lambda$ defined in~\eqref{defn:Al}, and therefore, condition~\eqref{hp:tar-c'} is satisfied in $[\al,T]$. From Proposition~\ref{prop:tarama} we deduce
\begin{eqnarray}
    \ul'(t)^2+\lambda^2 b(t)^2 c(t)^2\ul(t)^2
    & \leq &
    5\frac{c(t)}{c(\al)}\cdot
    \exp(\Phi_{2,3}(\lambda))\cdot
    \frac{b(t)^2}{b(\al)^2}
    \nonumber
    \\
    & &
    \cdot
    \left[\ul'(\al)^2+
    \lambda^2 b(\al)^2 c(\al)^2 \ul(\al)^2\right]
    \label{est:c2-aT-local}
\end{eqnarray}
for every $t\in[\al,T]$, where
\begin{equation}
    \Phi_{2,3}(\lambda):=
    \frac{3}{\lambda}
    \int_{\al}^{T}
    \frac{1}{b(t)}
    \left(\frac{|c''(t)|}{c(t)^2}+\frac{c'(t)^2}{c(t)^3}\right)dt.
    \nonumber
\end{equation}

Estimating the last factor in the right-hand side of~\eqref{est:c2-aT-local} by means of~\eqref{est:c2-sa} with $t=\al$, from (\ref{est:c2-aT-local}) we deduce that
\begin{eqnarray}
    \ul'(t)^2+\lambda^2 b(t)^2 c(t)^2\ul(t)^2
    & \leq &
    \frac{5\nu_2}{\nu_1}
    \exp(\Phi_{2,1}(\lambda))\cdot
    \exp(\Phi_{2,2}(\lambda))\cdot
    \exp(\Phi_{2,3}(\lambda))\cdot
    \frac{b(t)^2}{b(\tl)^2}
    \nonumber
    \\
    & &
    \cdot
    \left[\ul'(\tl)^2+
    \lambda^2 b(\tl)^2 c(\ssl)^2 \ul(\tl)^2\right]
    \label{est:c2-aT}
\end{eqnarray}
for every $t\in[\al,T]$.

By the definition of $(\ssl,\al)$, and by the estimates above, we have
\begin{equation}
    \Phi_{2,1}(\lambda)+\Phi_{2,2}(\lambda)+\Phi_{2,3}(\lambda)=
    \psi_{2,\lambda}(\ssl,\al)
    =
    \Phi_2(\lambda,T).
    \nonumber
\end{equation}
Combining this inequality with~\eqref{est:c2-ts}, \eqref{est:c2-sa}, and~\eqref{est:c2-aT}, we obtain~\eqref{est:c2-after} for every $t\in[\tl,T]$.

Due to~\eqref{est:c2-kow} and~\eqref{est:c2-after}, the function $v(t):=\ul(t)$ satisfies the assumptions of Proposition~\ref{prop:mix} with
\[
    H_1:=\exp\left(\frac{\nu_2}{\nu_1}\right),
    \qquad
    H_2:=
    \frac{5\nu_2}{\nu_1}
    \exp(\Phi_2(\lambda,T)).
\]
With these choices, estimate~\eqref{th:est-kow-hyp} gives
\begin{eqnarray*}
    \ul'(t)^2+\lambda^2\ul(t)^2
    & \leq &
    15\left(\frac{\nu_2}{\nu_1}\right)^3
    \exp\left(\frac{\nu_2}{\nu_1}\right)\cdot
    \exp(\Phi_2(\lambda,T))\cdot
    \left(1+\nu_2^2b(T)^2\right)
    \\
    & &
    \cdot
    \left[(1+\lambda^2\tl^2)\ul'(0)^2+
    \lambda^2\ul(0)^2\right]
\end{eqnarray*}
for every $t\in[0,T]$, which is exactly~\eqref{th:C2-theory}.
\qed


\subsection{Proof of Corollary~\ref{cor:natural-scale-c2}}

We divide the proof into several steps.

\paragraph{\textmd{\textit{The auxiliary splitting point $\al$}}}

We define
\begin{equation}
    \al:=
    \inf\left\{t\in[\tl,T]:
    \frac{\nu_1^2\lambda}{M}\geq
    \frac{|\log B(t)|}{B(t)}
    \right\}.
    \label{defn:al-B}
\end{equation}
This definition is meaningful under the first condition in~(\ref{hp:lambda-c2}), because that condition implies that the set in~(\ref{defn:al-B}) is nonempty.

We claim that
\begin{equation}
    A_\lambda\leq \al.
    \label{est:Al<=al}
\end{equation}
Indeed, since $B(t)\leq1$ for every $t\in(0,T]$, and the function $\sigma\mapsto|\log\sigma|/\sigma$ is decreasing in $(0,1]$, we have
\[
    \frac{|\log B(t)|}{B(t)}
    \leq
    \frac{\nu_1^2\lambda}{M}
    \qquad
    \forall t\in[\al,T].
\]
Therefore, by the assumption on $c$ and $c'$, it follows that
\[
    |c'(t)|
    \leq
    M\frac{b(t)}{B(t)}|\log B(t)|
    \leq
    \nu_1^2\lambda b(t)
    \leq
    \lambda b(t)c(t)^2
    \qquad
    \forall t\in[\al,T],
\]
and this proves~(\ref{est:Al<=al}).

\paragraph{\textmd{\textit{Upper bound for $B(\tl)$}}}

We claim that
\begin{equation}
    B(\tl)\leq\frac{1}{\nu_2\lambda}.
    \label{est:Btl}
\end{equation}
Indeed, by the monotonicity of $b$ we have
\begin{equation}
    B(t)\leq tb(t)
    \qquad
    \forall t\in[0,T].
    \nonumber
\end{equation}
Therefore, by the definition of $\tl$,
\begin{equation}
    B(t)\leq tb(t)<\frac{1}{\nu_2\lambda}
    \qquad
    \forall t\in[0,\tl),
    \nonumber
\end{equation}
and passing to the limit as $t\to\tl^-$ gives~(\ref{est:Btl}).

\paragraph{\textmd{\textit{Strict bounds for $\al$}}}

We now show that, under the conditions in~(\ref{hp:lambda-c2}), one has
\begin{equation}
    \tl<\al<T
    \nonumber
\end{equation}
and
\begin{equation}
    \frac{|\log B(\al)|}{B(\al)}
    =
    \frac{\nu_1^2\lambda}{M}.
    \label{eqn:Bal}
\end{equation}

The fact that $\al<T$ follows from the strict inequality in the first condition in~(\ref{hp:lambda-c2}), and the continuity of $B$. In order to prove that $\tl<\al$, we argue by contradiction and assume that $\al=\tl$. Since the defining set in~(\ref{defn:al-B}) is closed, this gives
\begin{equation}
    \frac{\nu_1^2\lambda}{M}\geq
    \frac{|\log B(\tl)|}{B(\tl)}.
    \label{est:log-Btl}
\end{equation}
On the other hand, the second condition in~(\ref{hp:lambda-c2}) implies in particular that $\nu_2\lambda>1$. Hence, (\ref{est:Btl}) gives
\[
    B(\tl)\leq\frac{1}{\nu_2\lambda}<1.
\]
Since $\sigma\mapsto |\log\sigma|/\sigma$ is decreasing in $(0,1]$, it follows that
\begin{equation}
    \frac{|\log B(\tl)|}{B(\tl)}
    \geq
    \nu_2\lambda\log(\nu_2\lambda).
    \nonumber
\end{equation}
Combined with~(\ref{est:log-Btl}), this implies
\[
    \frac{\nu_1^2}{M}\geq
    \nu_2\log(\nu_2\lambda),
\]
which contradicts the second condition in~(\ref{hp:lambda-c2}). This proves that $\tl<\al$.

Since $\al$ lies in the open interval $(\tl,T)$, identity~(\ref{eqn:Bal}) follows from the definition of $\al$ and the continuity of $B$.

\paragraph{\textmd{\textit{Bounds for $B(\al)$}}}

We prove that
\begin{equation}
    \frac{1}{\nu_2\lambda}\leq B(\al)\leq 1.
    \label{est:Bal}
\end{equation}
The upper bound is part of the assumptions of the corollary, because $B(\al)\leq B(T)\leq 1$. Let us prove the lower bound. If
\[
    B(\al)<\frac{1}{\nu_2\lambda},
\]
then, using again the monotonicity of $\sigma\mapsto|\log\sigma|/\sigma$ in $(0,1]$, we would obtain
\[
    \frac{\nu_1^2\lambda}{M}
    =
    \frac{|\log B(\al)|}{B(\al)}
    >
    \nu_2\lambda\log(\nu_2\lambda),
\]
which contradicts the second condition in~(\ref{hp:lambda-c2}). This proves~(\ref{est:Bal}).

As a consequence of~(\ref{est:Bal}), since $\nu_2\lambda>1$, we also have
\begin{equation}
    |\log B(\al)|\leq\log(\nu_2\lambda).
    \label{est:Bal-log}
\end{equation}

\paragraph{\textmd{\textit{Conclusion}}}

The pair $(\al,\al)$ is an admissible competitor in the minimization problem~(\ref{defn:Phi2}), because of~(\ref{est:Al<=al}). Therefore,
\begin{eqnarray}
    \Phi_2(\lambda,T)
    & \leq &
    \psi_{2,\lambda}(\al,\al)
    \nonumber
    \\
    & \leq &
    \frac{2\nu_2}{\nu_1}(\nu_2-\nu_1)
    \lambda B(\al)
    \nonumber
    \\
    & &
    \mbox{}+
    \frac{3M^2(\nu_1+1)}{\nu_1^3\lambda}
    \int_{\al}^{T}
    b(t)\frac{\log^2 B(t)}{B(t)^2}\,dt.
    \label{est:Phi2-comp}
\end{eqnarray}
Indeed, in the first term in the definition of $\psi_{2,\lambda}(\al,\al)$ we used
\[
    \osc(c,(0,\al))\leq \nu_2-\nu_1
    \qquad\quad\text{and}\quad\qquad
    \int_{\tl}^{\al}b(t)\,dt\leq B(\al),
\]
while in the third term we used the assumptions on $c'$ and $c''$, together with $c(t)\geq\nu_1$, to get
\[
    \frac{|c''(t)|}{c(t)^2}+\frac{c'(t)^2}{c(t)^3}
    \leq
    M^2\cdot\frac{\nu_1+1}{\nu_1^3}\cdot
    \frac{b(t)^2}{B(t)^2}\log^2 B(t).
\]

Let us estimate the two terms in the right-hand side of~(\ref{est:Phi2-comp}). From~(\ref{eqn:Bal}) and~(\ref{est:Bal-log}) we obtain
\begin{equation}
    \lambda B(\al)
    =
    \frac{M}{\nu_1^2}
    |\log B(\al)|
    \leq
    \frac{M}{\nu_1^2}\log(\nu_2\lambda).
    \label{est:first-term-c2}
\end{equation}

As for the second term, the change of variable $x=B(t)$ gives
\begin{equation}
    \int_{\al}^{T}
    b(t)\frac{\log^2 B(t)}{B(t)^2}\,dt
    =
    \int_{B(\al)}^{B(T)}
    \frac{\log^2 x}{x^2}\,dx
    \leq
    \int_{B(\al)}^{+\infty}
    \frac{\log^2 x}{x^2}\,dx.
    \nonumber
\end{equation}
Since
\begin{equation}
    \int_{\alpha}^{+\infty}\frac{\log^2 x}{x^2}\,dx
    =
    \frac{\log^2\alpha+2\log\alpha+2}{\alpha}
    \leq
    \frac{2\log^2\alpha+3}{\alpha}
    \qquad
    \forall \alpha>0,
    \nonumber
\end{equation}
we deduce that
\begin{equation}
    \int_{\al}^{T}
    b(t)\frac{\log^2 B(t)}{B(t)^2}\,dt
    \leq
    \frac{2\log^2 B(\al)+3}{B(\al)}.
    \nonumber
\end{equation}
Using~(\ref{eqn:Bal}), (\ref{est:Bal}), and~(\ref{est:Bal-log}), we conclude that
\begin{eqnarray}
    \frac{1}{\lambda}
    \int_{\al}^{T}
    b(t)\frac{\log^2 B(t)}{B(t)^2}\,dt
    & \leq &
    \frac{1}{\lambda}
    \left(
    \frac{2|\log B(\al)|^2}{B(\al)}
    +
    \frac{3}{B(\al)}
    \right)
    \nonumber
    \\[0.5ex]
    & = &
    \frac{1}{\lambda}
    \left(
    \frac{2\nu_1^2\lambda}{M}
    |\log B(\al)|
    +
    \frac{3}{B(\al)}
    \right)
    \nonumber
    \\[0.5ex]
    & \leq &
    \frac{2\nu_1^2}{M}\log(\nu_2\lambda)
    +
    3\nu_2.
    \label{est:second-term-c2}
\end{eqnarray}

Plugging~(\ref{est:first-term-c2}) and~(\ref{est:second-term-c2}) into~(\ref{est:Phi2-comp}), we obtain
\begin{eqnarray*}
    \Phi_2(\lambda,T)
    & \leq &
    \frac{2M\nu_2(\nu_2-\nu_1)}{\nu_1^3}
    \log(\nu_2\lambda)
    \nonumber
    \\
    & &
    \mbox{}+
    \frac{6M(\nu_1+1)}{\nu_1}
    \log(\nu_2\lambda)
    +
    \frac{9\nu_2M^2(\nu_1+1)}{\nu_1^3},
\end{eqnarray*}
which is the desired estimate.
\qed


\subsubsection*{\centering Acknowledgments}

The first two authors are members of the Italian {\selectlanguage{italian}%
``Gruppo Nazionale per l'Analisi Matematica, la Probabilit\`{a} e le loro Applicazioni'' (GNAMPA) of the ``Istituto Nazionale di Alta Matematica'' (INdAM)}, and acknowledge the MIUR Excellence Department Project awarded to the Department of Mathematics of the University of Pisa (CUP I57G22000700001).


\begin{thebibliography}{10}
\providecommand{\url}[1]{\texttt{#1}}
\providecommand{\urlprefix}{URL }
\providecommand{\selectlanguage}[1]{\relax}
\providecommand{\eprint}[2][]{\url{#2}}

\bibitem{1983-AlexandrianYagdjian}
\textsc{G.~R. Alexandrian}, \textsc{K.~H. Yagdjian}.
\newblock On the correctness of the {C}auchy problem and propagation of the
  singularities for the differential equations with real multiple
  characteristics.
\newblock \emph{Soviet J. Contemporary Math. Anal.} \textbf{18} (1983), no.~2,
  53--68.

\bibitem{2008-CicognaniHirosawaReissig}
\textsc{M.~Cicognani}, \textsc{F.~Hirosawa}, \textsc{M.~Reissig}.
\newblock Loss of regularity for {$p$}-evolution type models.
\newblock \emph{J. Math. Anal. Appl.} \textbf{347} (2008), no.~1, 35--58.

\bibitem{1979-DGCS}
\textsc{F.~Colombini}, \textsc{E.~De~Giorgi}, \textsc{S.~Spagnolo}.
\newblock Sur les \'{e}quations hyperboliques avec des coefficients qui ne
  d\'{e}pendent que du temps.
\newblock \emph{Ann. Scuola Norm. Sup. Pisa Cl. Sci. (4)} \textbf{6} (1979),
  no.~3, 511--559.

\bibitem{2002-SNS-ColDSaKin}
\textsc{F.~Colombini}, \textsc{D.~Del~Santo}, \textsc{T.~Kinoshita}.
\newblock Well-posedness of the {C}auchy problem for a hyperbolic equation with
  non-{L}ipschitz coefficients.
\newblock \emph{Ann. Sc. Norm. Super. Pisa Cl. Sci. (5)} \textbf{1} (2002),
  no.~2, 327--358.

\bibitem{2003-BSM-ColDSaRei}
\textsc{F.~Colombini}, \textsc{D.~Del~Santo}, \textsc{M.~Reissig}.
\newblock On the optimal regularity of coefficients in hyperbolic {C}auchy
  problems.
\newblock \emph{Bull. Sci. Math.} \textbf{127} (2003), no.~4, 328--347.

\bibitem{1983-CJS}
\textsc{F.~Colombini}, \textsc{E.~Jannelli}, \textsc{S.~Spagnolo}.
\newblock Well-posedness in the {Gevrey} classes of the {Cauchy} problem for a
  non-strictly hyperbolic equation with coefficients depending on time.
\newblock \emph{Ann. Scuola Norm. Sup. Pisa Cl. Sci.} \textbf{10} (1983),
  no.~2, 291--312.

\bibitem{2007-DelSantoKinoshitaReissig}
\textsc{D.~Del~Santo}, \textsc{T.~Kinoshita}, \textsc{M.~Reissig}.
\newblock Energy estimates for strictly hyperbolic equations with low
  regularity in coefficients.
\newblock \emph{Differential Integral Equations} \textbf{20} (2007), no.~8,
  879--900.

\bibitem{gg:CDSR-optimal}
\textsc{M.~Ghisi}, \textsc{M.~Gobbino}.
\newblock Finite vs infinite derivative loss for abstract wave equations with
  singular time-dependent propagation speed.
\newblock \emph{Bull. Sci. Math.} \textbf{166} (2021), 102918, 41.

\bibitem{gg:OptDerLoss}
\textsc{M.~Ghisi}, \textsc{M.~Gobbino}.
\newblock Optimal derivative loss for abstract wave equations.
\newblock \emph{Math. Ann.} \textbf{386} (2023), 455--494.

\bibitem{gg:residual}
\textsc{M.~Ghisi}, \textsc{M.~Gobbino}.
\newblock Three examples of residual pathologies.
\newblock \emph{Jpn. J. Math.} \textbf{18} (2023), no.~1, 49--113.

\bibitem{gg:summa-theologica}
\textsc{M.~Ghisi}, \textsc{M.~Gobbino}.
\newblock Fifty years of wave equations with time-dependent propagation speeds:
  A variational perspective and new frontiers, 2026.
\newblock \urlprefix\url{https://arxiv.org/abs/2608.24286}.
\newblock \eprint{2608.24286}.

\bibitem{2003-Hirosawa}
\textsc{F.~Hirosawa}.
\newblock On the {C}auchy problem for second order strictly hyperbolic
  equations with non-regular coefficients.
\newblock \emph{Math. Nachr.} \textbf{256} (2003), 29--47.

\bibitem{2004-HirosawaReissig}
\textsc{F.~Hirosawa}, \textsc{M.~Reissig}.
\newblock Well-posedness in {S}obolev spaces for second-order strictly
  hyperbolic equations with nondifferentiable oscillating coefficients.
\newblock \emph{Ann. Global Anal. Geom.} \textbf{25} (2004), no.~2, 99--119.

\bibitem{1974-IvrijPetkov}
\textsc{V.~Y. Ivrii}, \textsc{V.~M. Petkov}.
\newblock Necessary conditions for the correctness of the {C}auchy problem for
  non-strictly hyperbolic equations.
\newblock \emph{Uspekhi Mat. Nauk} \textbf{29} (1974), no. 5(179), 3--70.

\bibitem{2005-KinoshitaReissig}
\textsc{T.~Kinoshita}, \textsc{M.~Reissig}.
\newblock About the loss of derivatives for strictly hyperbolic equations with
  non-{L}ipschitz coefficients.
\newblock \emph{Adv. Differential Equations} \textbf{10} (2005), no.~2,
  191--222.

\bibitem{1970-Oleijnik}
\textsc{O.~A. Oleinik}.
\newblock On the {C}auchy problem for weakly hyperbolic equations.
\newblock \emph{Comm. Pure Appl. Math.} \textbf{23} (1970), no.~4, 569--586.

\bibitem{1980-TaniguchiTozaki}
\textsc{K.~Taniguchi}, \textsc{Y.~Tozaki}.
\newblock A hyperbolic equation with double characteristics which has a
  solution with branching singularities.
\newblock \emph{Math. Japon.} \textbf{25} (1980), no.~3, 279--300.

\bibitem{1995-Tarama}
\textsc{S.~Tarama}.
\newblock On the second order hyperbolic equations degenerating in the infinite
  order. {Example}.
\newblock \emph{Math. Japon.} \textbf{42} (1995), no.~3, 523--533.

\bibitem{2007-EJDE-Tarama}
\textsc{S.~Tarama}.
\newblock Energy estimate for wave equations with coefficients in some {B}esov
  type class.
\newblock \emph{Electron. J. Differential Equations}  (2007), No. 85, 12.

\bibitem{1980-Yagdjian}
\textsc{K.~A. Yagdjian}.
\newblock On correctness of the {C}auchy problem for non-strictly hyperbolic
  equations.
\newblock \emph{Soviet J. Contemporary Math. Anal.} \textbf{15} (1980), no.~6,
  54--65.

\bibitem{1997-Yagdjian}
\textsc{K.~A. Yagdjian}.
\newblock \emph{The {C}auchy Problem for Hyperbolic Operators: Multiple
  Characteristics. Micro-Local Approach}, \emph{Mathematical Topics},
  volume~12.
\newblock Akademie-Verlag, Berlin, 1997.

\bibitem{Yamazaki1990}
\textsc{T.~Yamazaki}.
\newblock Second order hyperbolic equations with time-dependent singularity or
  degeneracy.
\newblock \emph{Hokkaido Math. J.} \textbf{19} (1990), no.~1, 125--149.

\end{thebibliography}

\label{NumeroPagine}

\end{document}